\documentclass[preprint,12pt]{elsarticle}

\usepackage{amssymb}   
\usepackage{amsthm}    
\usepackage{amsmath}   
\usepackage{cases}
\usepackage{stmaryrd}  
\usepackage{titletoc}  
\usepackage{mathrsfs}  

\theoremstyle{plain}
\newtheorem{thm}{Theorem}[section]
\newtheorem{cor}[thm]{Corollary}
\newtheorem{lem}[thm]{Lemma}
\newtheorem{prop}[thm]{Proposition}

\theoremstyle{definition}
\newtheorem{defn}{Definition}[section]

\newtheorem{rmk}{Remark}[section]

\newcommand{\R}{\mathbb{R}}
\newcommand{\E}{\mathbb{E}}
\newcommand{\bR}{\mathbb{R}}

\newcommand{\bP}{\mathbb{P}}

\newcommand{\cM}{\mathcal{M}}

\newcommand{\sF}{\mathscr{F}}
\newcommand{\sP}{\mathscr{P}}

\newcommand{\A}{\textbf{A}}
\newcommand{\B}{\textbf{B}}
\newcommand{\F}{\textbf{F}}
\newcommand{\DP}{\textbf{P}}
\newcommand{\SP}{\textbf{SP}}

\newcommand{\eps}{\varepsilon}
\newcommand{\vf}{\varphi}

\newcommand{\la}{\langle}
\newcommand{\ra}{\rangle}

\newcommand{\ptl}{\partial}
\newcommand{\rrow}{\rightarrow}

\makeatletter
   
   \@addtoreset{equation}{section}
\makeatother

\begin{document}

\begin{frontmatter}



\title{$W^{m,p}$-Solution ($p\ge 2$) of Linear Degenerate Backward Stochastic Partial Differential Equations
in the Whole Space\tnoteref{label1}}

\tnotetext[label1]{Supported by NSFC Grant \#10325101, by Basic Research Program of China (973 Program) Grant \#2007CB814904, by the Science
Foundation for Ministry of Education of China Grant \#20090071110001, by Program for Changjiang Scholars, and by WCU (World Class University)
Program through the Korea Science and Engineering Foundation funded by the Ministry of Education, Science and Technology
(R31-2009-000-20007).}

\author[label2]{Kai Du}
\ead{kdu@fudan.edu.cn}

\author[label2,label3]{Shanjian Tang}
\ead{sjtang@fudan.edu.cn}

\author[label2]{Qi Zhang}
\ead{qzh@fudan.edu.cn}

\address[label2]{Department of Finance and Control Sciences, School of
Mathematical Sciences, Fudan University, Shanghai 200433, China.}

\address[label3]{Graduate Department of Financial Engineering, Ajou
University, San 5, Woncheon-dong, Yeongtong-gu, Suwon, 443-749, Korea.}


\begin{abstract}
In this paper, we consider the backward Cauchy problem of linear degenerate  stochastic partial differential equations. We obtain the
existence and uniqueness results in Sobolev space $L^p(\Omega; C([0,T];W^{m,p}))$ with both $m\geq 1$ and  $p\geq 2$ being arbitrary, without
imposing the symmetry condition for the coefficient $\sigma$ of the gradient of the second unknown---which was introduced by Ma and
Yong~\cite[Prob. Theor. Relat. Fields 113 (1999)]{MaYo99} in the case of $p=2$. To illustrate the application, we give a maximum principle for
optimal control of degenerate stochastic partial differential equations.
\end{abstract}


\end{frontmatter}


\section{Introduction}

Let $(\Omega,\mathscr{F},\{\mathscr{F}_{t}\}_{t\geq 0},P)$ be a complete filtered probability space on which a $d'$-dimensional Wiener process
$W=\{W_t;t\geq 0\}$ is defined such that $\{\mathscr{F}_{t}\}_{t\geq 0}$ is the natural filtration generated by $W$, augmented by all the
$P$-null sets in $\mathscr{F}$. We denote by $\mathscr{P}$ the predictable $\sigma$-algebra associated with $\{\mathscr{F}_{t}\}_{t\geq 0}$.
Let $T$ be a fixed positive number.

We consider the backward Cauchy problem for the following linear  stochastic partial differential equation (BSPDE in short)
\begin{eqnarray}\label{dtz6}
    du(t,x) &=& -\sum_{i,j=1}^{d}\sum_{k=1}^{d'}
    \Big\{\big[a^{ij}(t,x)u_{x^j}(t,x)
    +\sigma^{ik}(t,x)q^{k}(t,x)\big]_{x^i}+b^{i}(t,x)u_{x^i}(t,x)\nonumber\\
    &&+\,c(t,x)u(t,x)+\nu^k(t,x) q^k(t,x) +
    f(t,x) \Big\}\,dt
    + \sum_{k=1}^{d'}q^k(t,x) \,dW^k_t,\\
   && \qquad (t,x)\in [0,T)\times\R^d, \quad W_t \triangleq (W^1,\cdots, W^{d'}_t),\quad t\geq 0 \nonumber
\end{eqnarray}
and the terminal condition
\begin{eqnarray}\label{TC} u(T,x)=\vf(x), \quad x\in \R^d. \end{eqnarray} They are a mathematically natural extension of
backward stochastic differential equations (BSDEs) (see e.g. \cite{KPQ97,PaPe90}), whose solution consists of a pair of random fields $(u,q)$.

BSPDEs arise in many applications of probability theory and stochastic processes. For instance, in the optimal control of SDEs with partial
information or more generally of stochastic parabolic PDEs, as adjoint equations of Duncan-Mortensen-Zakai filtering equations (see e.g.
\cite{Bens83,Haus87,NaNi90,Tang98a,Tang98b,Zhou93}) to formulate the stochastic maximum principle for the optimal control, and in the
formulation of the stochastic Feynman-Kac formula (see e.g. \cite{MaYo97}) in mathematical finance. A class of fully nonlinear BSPDEs, the
so-called backward stochastic Hamilton-Jacobi-Bellman equations, appear naturally in the dynamic programming theory of controlled
non-Markovian processes (see \cite{CPY09,EnKa09,Peng92a}). For more aspects of BSPDEs, one can refer to
\cite{BRT03,DuMe10,HMY02,Naga04,Tang05,TaZh09,Tess96}.

In the case of the following \emph{super-parabolicity}:  the matrix
\begin{equation}\label{SuperParab}
(A^{ij}):=\Big(a^{ij}-\frac{1}{2}\sum_{k=1}^{d'}\sigma^{ik}\sigma^{jk} \Big)_{d\times d} \geq \delta I_{d}\quad \hbox{ \rm for some }
\delta>0,
\end{equation} the theory of existence, uniqueness and regularity for solutions to BSPDE~\eqref{dtz6}-\eqref{TC} is rather complete now.
See among others~\cite{Huss09,Bens83,DuMe10,DQT10,Tess96,Zhou92}.  On the contrary, very few studies are devoted to
BSPDEs~\eqref{dtz6}-\eqref{TC} for the case of the \emph{degenerate parabolicity}:
\begin{equation}\label{deg}
   (A^{ij})\geq 0.
\end{equation}
 Zhou~\cite{Zhou92} seemed to be the first to study degenerate BSPDEs, but under the assumption that the matrix $\sigma=(\sigma^{ik})_{d\times d'}$
 vanishes. He ended up with commenting (see ~\cite[page 340]{Zhou92}) that
``it remains a challenging open problem to solve the adjoint equations of degenerate SPDEs in which diffusion terms contain first-order
differential operators"  (in a dually equivalent term, the adjoint equations as BSPDEs  contain the first order derivative of $q$, that is
$\sigma$ does not vanish). Seven years later, by introducing the following symmetry condition
\begin{equation}\label{sym}
    \sum_{k=1}^{d'}\sigma^{ik}{\partial \sigma^{jk}\over \partial x^l}=\sum_{k=1}^{d'}\sigma^{jk}{\partial \sigma^{ik}\over \partial x^l},
    \quad i,j,l=1,\ldots,d,
\end{equation}
Ma and Yong~\cite{MaYo99} made a first step towards solution of the open challenging problem. In particular, assuming that the coefficient
$\sigma$ is invariant with the state variable $x$, which certainly satisfies the symmetry condition~\eqref{sym},  Hu, Ma and Yong~\cite{HMY02}
further solved a semi-linear degenerate BSPDE for the case of $d=1$. Note that counterexamples to the symmetry condition~\eqref{sym} are many,
and here are three ones for the case of $d=d'=2$ and $x:=(x^1,x^2)\in \R^2$:
\begin{equation}\label{counterex}
\sigma (x) = \left(\begin{array}{cc} \sin(x^1+x^2)& \cos(x^1+x^2)\\
\cos(x^1+x^2)& -\sin(x^1+x^2) \end{array}\right),
\end{equation}
\begin{equation}\label{counterex}
 \sigma (x) = \left(\begin{array}{cc} {1\over \sqrt{1+|x|^2}}& 1\\
0& -{1\over \sqrt{1+|x|^2}} \end{array}\right),
\end{equation}
and
\begin{equation}\label{counterex}
\sigma (x) = \left(\begin{array}{cc} \sin(|x|)& \cos(|x|)\\
\cos(|x|)& -\sin(|x|) \end{array}\right) \quad \hbox{ \rm with } |x|:=\sqrt{(x^1)^2+(x^2)^2}.
\end{equation}

The preceding three works are basically restricted within the framework of the space $H^{m}=W^{m,2}$ with respect to the spatial variable
$x\in \R^d$, with the only exception that Hu, Ma and Yong~\cite[Theorem 3.2, page 394]{HMY02} includes an $L^p$ estimate for positive even
number $p$. Using the method of stochastic flows, imposing neither any symmetry condition on $\sigma$ nor any restriction on the dimension of
$\R^d$, Tang~\cite{Tang05} represented the solution of a system of semi-linear degenerate BSPDEs via the solutions of a family of FBSDEs and
then established the existence and uniqueness of classical solutions to semi-linear systems of degenerate BSPDEs, but at a cost of assuming
differentiability of higher orders in $x$ on the coefficients. Some $W^{m,p}$ estimates are also given via the method of stochastic flows in
Tang~\cite{Tang03} for $m=1,2$ and $p\ge 2$. The last two works strongly suggest that the symmetry condition~\eqref{sym} be superfluous.

In this paper, we obtain the existence, uniqueness and regularity of generalized (or weak) solutions of linear degenerate BSPDEs under a quite
general setting (see Theorem~\ref{thm:main}), in particular without imposing the symmetry condition~\eqref{sym}. Besides, we give a
$W^{m,p}$-estimate with $m\ge 1$ for any $p\geq2$. Our approach inherits  the  spirit of Krylov and Rozowskii~\cite{KrRo82} in the solution of
degenerate stochastic PDEs, and invokes  the real analysis lemma of Oleinik~\cite{Olei67} in the study of linear degenerate parabolic PDEs  so
as to derive a prior estimates for degenerate BSPDEs. However, we have to develop some innovations to their calculus. The fundamental estimate
of Krylov and Rozowskii~\cite[Lemma 2.1, page 339]{KrRo82} is based on computing the quantity
$$
\sum_{|\alpha|\le m}G\left(|D^\alpha u|^2\right)
$$
for the convex function $G(s):=s^p, s\ge 0$. Repetition of their calculus to our BSPDEs~\eqref{eq:trans} finds the following new and trouble
term:
$$
\sum_{|\alpha|\le m, \beta\le \alpha} D^\alpha u D^\beta r G'\left(|D^\alpha u|^2\right),
$$
which involves both unknowns $u$ and $r$ of our BSPDEs~\eqref{eq:trans}. Using Cauchy-Schwarz inequality, we only get the following
$$
2D^\alpha u D^\beta r G'\left(|D^\alpha u|^2\right)\le  G'\left(|D^\alpha u|^2\right) \left(\epsilon|D^\beta r|^2+\epsilon^{-1}|D^\alpha
u|^2\right), \quad \beta < \alpha,
$$
and it is difficult to see unless $p=1$ that it is dominated by the following
$$
\sum_{|\alpha|\le m} \left(|D^\alpha u|^2+ |D^\alpha r|^2\right) G'\left(|D^\alpha u|^2\right)
$$
due to the fact that $G'\left(|D^\alpha u|^2\right)$ varies with $\alpha$ for $p>1$.
 To overcome this difficulty, we instead compute the following
\begin{equation}
G\left(\sum_{|\alpha|\le m}|D^\alpha u|^2\right),
\end{equation}
which leads to the following term
$$
G'\left(\sum_{|\alpha|\le m}|D^\alpha u|^2\right)\sum_{|\alpha|\le m, \beta\le \alpha} D^\alpha u D^\beta r.
$$
Since the derivative $G'\left(\sum_{|\alpha|\le m}|D^\alpha u|^2\right)$ is invariant with $\alpha$ for any $p\ge 1$, the sum is dominated by
\begin{equation}
G'\left(\sum_{|\alpha|\le m}|D^\alpha u|^2\right)\sum_{|\alpha|\le m} \left(|D^\alpha u|^2+ |D^\alpha r|^2\right).
\end{equation}
In this way, we introduce a series of innovational computations and develop some new techniques. For example, a localization method and some
calculation techniques in application of It\^o formula (see Remark \ref{dtz5} (i)) are given to deal with the difficult and subtle degenerate
parabolicity condition.

As we know, BSPDEs arising from many applications are degenerately parabolic rather than being super-parabolic (see, e.g.
\cite{Peng92a,MaYo97,Tang98a,Tang98b,EnKa09}). Optimal control  of SPDEs is solved in Section 4, so as to illustrate the application of our
new result on degenerate parabolic BSPDE.


%

The rest of the paper is organized as follows. In Section 2, we introduce some notations and state our main result (Theorem \ref{thm:main}).
In Section 3, we prove the main theorem, after some preparations starting from the super-parabolic BSPDE with smooth coefficients. In Section
4, a maximum principle for optimal control of degenerate SPDEs is formulated  as an application of degenerate parabolic BSPDEs. In Sections 5
and 6, we give the proofs for two important lemmas which are used in the proof of Theorem~\ref{thm:main}.

\section{Preliminaries and the main result}

Let $d$ and $d'$ be two positive integers, and  $\R^d$ be $d$-dimensional Euclidean space whose squared norm  is denoted by $\|\cdot\|$.
Consider the following linear BSPDE
\begin{equation}\label{eq:main}
    du=-(a^{ij}u_{x^i x^j}+b^{i}u_{x^i}+cu+\sigma^{ik}q^{k}_{x^i}+\nu^k q^k +
    f)\,dt + q^k \,dW^k_t
\end{equation}
with the terminal condition
\begin{equation}\label{con:terminal}
  u(T)=\vf.
\end{equation}
Here,  the coefficients $a^{ij}(t,x)=a^{ji}(t,x)$, $b^{i}(t,x)$, $\sigma^{ik}(t,x), \nu^{k}(t,x)$ for $i,j=1,\dots,d$ and $k=1,\dots,d'$, and
$c(t,x), f(t,x)$ and $\vf(x)$ (as a rule, the argument $\omega$ is omitted), $(\omega,t,x)\in \Omega\times[0,T]\times\bR^d$, are given and
measurable. The summation convention is in force for repeated indices throughout this paper.

To define the solution of BSPDE \eqref{eq:main}-\eqref{con:terminal}, we need more notations.

$\bullet$ Denote $u_{x} = \nabla u = (u_{x^1},\cdots,u_{x^d})$ for $x\in \R^d$. For any multi-index $\alpha=(\alpha_1,\dots,\alpha_d)$, denote
$$D^{\alpha}=\bigg(\frac{\ptl}{\ptl x^1}\bigg)^{\alpha_1}
\bigg(\frac{\ptl}{\ptl x^2}\bigg)^{\alpha_2} \cdots \bigg(\frac{\ptl}{\ptl
x^d}\bigg)^{\alpha_d}$$ and $|\alpha|=\alpha_1+\cdots +\alpha_d$.

$\bullet$ Denote by $C_{0}^{\infty}$ and $C^n$ the set of infinitely
differentiable real functions of compact support defined on $\R^d$ and the
set of $n$ times continuously differentiable functions on $\R^d$ such that
$$\|u\|_{C^n} := \sum_{|\alpha|\leq n} \sup_{x\in \R^d} |D^{\alpha}u(x)| <
\infty.$$

$\bullet$ For $p>1$ and an integer $m\geq 0$, denote by
$W^{m,p}=W^{m,p}(\bR^d;\bR)$ the Sobolev space of real functions on
$\bR^d$ with a finite norm
$$\|u\|_{m,p}=\bigg(\sum_{|\alpha|\leq
m}\int_{\bR^d}|D^{\alpha}u|^{p}\,dx\bigg) ^{\frac{1}{p}},$$ where $\alpha$ is
a multi-index. In particular, $W^{0,p}=L^p$. It is well known that
$W^{m,2}$ is a Hilbert space and denote its inner product by $\la \cdot,\cdot
\ra_m$.

$\bullet$ For $p>1$ and integer $m\geq 0$, denote by $W^{m,p}(d')=W^{m,p}(\R^d;\R^{d'})$ the Sobolev space of $d'$-dimensional vector-valued
functions on $\bR^d$, equipped  with the norm $\|v\|_{m,p} = (\sum_{k=1}^{d'}\|v^{k}\|_{m,p}^{p})^{1/p}$.

$\bullet$ Denote by $L^{p}_{\sP}W^{m,p}$ (resp. $L^{p}_{\sP}W^{m,p}(d')$) the
space of all predictable process $u: \Omega\times[0,T]\rrow W^{m,p}$ (resp. $u:
\Omega\times[0,T]\rrow W^{m,p}(d')$) such that $u(\omega,t)\in W^{m,p}$ (resp.
$u(\omega,t)\in W^{m,p}(d')$) for a.e. $(\omega,t)$ and
$$\E\int_0^T\|u(t)\|_{m,p}^p \,dt <\infty.$$

$\bullet$ Denote by $L^{p}_{\sP}CW^{m,p}$ (resp. $L^{p}_{\sP}C_{w}W^{m,p}$)
the space of all predictable process $u: \Omega\times[0,T]\rrow W^{m,p}$
strongly (resp. weakly) continuous w.r.t. $t$ on $[0,T]$ for a.s.
$\omega$, such that
$$\E\sup_{t\leq T}\|u(t)\|_{m,p}^p < \infty.$$
\medskip

Given an integer $m \geq 1$, we shall use the following assumptions.

\medskip
($\A_m$)~ The given functions $a=(a^{ij})$, $b=(b^i)$, $c$, $\sigma=(\sigma^{ik})$ and $\nu=(\nu^{k})$ are $\mathscr{P} \times
B(\R^d)$-measurable with values in $\mathbb{S}^d$ (the set of real symmetric $d\times d$ matrices),
$\mathbb{R}^{d},\mathbb{R},\mathbb{R}^{d\times d'}$ and $\mathbb{R}^{d'}$, respectively. The coefficients $b^i,c,\nu^k$ and their derivatives
w.r.t. $x$ up to the order $m$, as well as $a^{ij}$ and $\sigma^{ik}$ up to the order $\max\{2,m\}$, are uniformly bounded by the positive
constant $K_m$.

\medskip
($\DP$)~ (\emph{parabolicity}) For each $(\omega,t,x)\in \Omega\times[0,T]\times\bR^d$,
$$\big[ 2 a^{ij}(t,x)-\sigma^{ik}\sigma^{jk}(t,x) \big]
\xi^i\xi^j\geq 0,~~\forall~\xi\in \bR^d.$$

\medskip
($\F_m$)~ The function $f\in L^{2}_{\sP}W^{m,2}$, the function $\vf$ is $\sF_T\times B(\bR^d)$-measurable with $\vf(\omega,\cdot)\in W^{m,2}$
for each $\omega$ and $\vf$ is $\sF_{T}$-measurable as a function on $\Omega$ with values in $W^{m,2}$.

\begin{defn}\label{def:sol}
  A pair of random fields $(u,q)\in L^{2}_{\sP}W^{1,2}\times
  L^{2}_{\sP}W^{0,2}(d')$ is called a generalized (or weak) solution of BSPDE
  \eqref{eq:main}-\eqref{con:terminal} if for each $\eta\in
  C^{\infty}_{0}$ and a.e. $t$,
  \begin{eqnarray}\label{eq:weaksol}\begin{split}
    \la u(t),\eta\ra_{0}=&\la\vf,\eta\ra_{0}+\int_{t}^{T}
    \Big[-\la a^{ij}(s)u_{x^i}(s)+\sigma^{jk}(s)q^{k}(s),\eta_{x^j}\ra_0\\
    &+\la (b^{i}(s)-a^{ij}_{x^j}(s))u_{x^i}(s)
    +c(s)u(s)
    +(\nu^k(s)-\sigma^{jk}_{x^j}(s))q^{k}(s)\\&+f(s),\eta\ra_0\Big]\,ds
    -\int_{t}^{T}\la q^{k}(s),\eta\ra_0 \,dW^{k}_{s}\ \ \ \bP\textrm{-a.s.}
    \end{split}
  \end{eqnarray}
\end{defn}


\begin{rmk}\label{rmk:2-1}
  If $(u,q)$ is a generalized solution of BSPDE
  \eqref{eq:main}-\eqref{con:terminal} and
  $(u,q) \in L^{2}_{\sP}W^{n,2} \times L^{2}_{\sP}W^{n-1,2}$,
  then $u$ has
  a modification $\widetilde{u}: \Omega\times[0,T]\rrow W^{n-1,2}$ which is
  continuous in $t$ for all $\omega$. Actually, this assertion
  is a backward version
  of Theorem 1.3.2 in \cite{KrRo81} and follows
  immediately from Lemma 3.1 in \cite{DuMe10}. For this, one needs
  take
  $H=W^{n-1,2}$, $V=W^{n,2}$ and $V^{*}=W^{n-2,2}$ in the referenced
  lemma and replace $\eta$ by $(1-\Delta)^{n-1}\eta$ in
\eqref{eq:weaksol} with $\langle\cdot,\cdot\rangle_0$ accordingly
changed into $\langle\cdot,\cdot\rangle_{n-1}$, due to the fact that
  $\la v,(1-\Delta)^{n-1}\eta \ra_{0} = \la v,\eta \ra_{n-1}
  ~~\textrm{for}~~v \in W^{n-1,2}$.
\end{rmk}
In the rest of this paper, we always take the continuous version of $u$.
\medskip

Now we state our main theorem.

\begin{thm}\label{thm:main}
  Let conditions \emph{($\A_m$)}, \emph{($\DP$)}, and
  \emph{($\F_m$)} be satisfied for given $m \geq 1$.
  Then BSPDE~\eqref{eq:main}-\eqref{con:terminal} has a unique generalized
  solution $(u,q)$ such that
    $$u \in L^{2}_{\sP}C_{w}W^{m,2}~~{\rm and}~~
    q + \nabla u \,\sigma \in L^{2}_{\sP}W^{m,2}(d'),$$ and
    for any integer $m_1 \in [0,m]$, we have the estimates
    \begin{eqnarray}\label{est:p=2}
      \nonumber \E \sup_{t \leq T} \|u(t)\|_{m_1,2}^{2}
      + \E\int_{0}^{T} \|(q + \nabla u\,\sigma)(t)\|_{m_1,2}^{2}\,dt
      ~~\\
      \leq~C \E\bigg( \|\vf\|_{m_1,2}^{2}
      + \int_{0}^{T} \|f(t)\|_{m_1,2}^{2}\,dt \bigg).
    \end{eqnarray}
    In addition, if
    $f \in L^{p}_{\sP}W^{m,p}$ and $\vf \in L^{p}(\Omega,\sF_{T},W^{m,p})$
    for $p \geq 2$, then $u \in L^{p}_{\sP}C_{w}W^{m,p}$, and
    for any integer $m_1 \in [0,m]$,
    \begin{equation}\label{est:p=p}
      \E \sup_{t \leq T} \|u(t)\|_{m_1,p}^{p}
      \leq C e^{Cp}\E\bigg( \|\vf\|_{m_1,p}^{p}
      + \int_{0}^{T} \|f(t)\|_{m_1,p}^{p}\,dt \bigg).
    \end{equation}
   Here and in the rest of this paper, $C$ is a
   generic constant which depends only on $d,d',K_m,m$ and $T$.
\end{thm}

\begin{rmk}
Theorem \ref{thm:main} asserts the existence and uniqueness of the solution of a degenerate BSPDE under quite general conditions, and improves
the relevant results of Zhou~\cite{Zhou92}, Ma and Yong~\cite{MaYo99} and Tang~\cite{Tang05}. It is a rather satisfactory solution of the
 open problem posed by~Zhou \cite{Zhou92}.
\end{rmk}

The proof of Theorem \ref{thm:main} is deferred to the next section. We now give some useful corollaries.
By Sobolev's imbedding theorem (see
e.g. \cite{Adam75}), we have

\begin{cor}\label{cor:001}
  Under conditions of Theorem \ref{thm:main}, \emph{(i)} if $mp > d$,
  then the first component $u$ is jointly continuous in $(t,x)$ a.s.;
  \emph{(ii)} if $2(m-2) > d$, then $(u,q)$ is a classical solution of BSPDE~\eqref{eq:main}-\eqref{con:terminal} which also reads
    \begin{eqnarray}\label{eq:2-001}
    \begin{split}
      u(t,x) ~=~& \vf(x) + \int_{t}^{T}
      \big(a^{ij}u_{x^i x^j}+b^{i}u_{x^i}+cu
      +\sigma^{ik}q^{k}_{x^i}+\nu^{k}q^{k}+f \big)(s,x)\,ds\\
      &-\int_{t}^{T} q^{k}(s,x)\,dW_{s}^{k}
    \end{split}
    \end{eqnarray}
    holds for all $x\in\R^d$, $t \in [0,T]$ on a $(t,x)$-independent set of the full measure.
\end{cor}

\begin{proof}
  We only need to prove assertion (ii). In view of Theorem \ref{thm:main}
  and Sobolev's imbedding theorem, we can select appropriate versions such
  that $u(t,\cdot)\in C^2(\R^d)$ a.s., and
  $q(t,\cdot)\in C^1(\R^d), f(t,\cdot),\vf(\cdot)\in C^2(\R^d)$
  for a.e. $(\omega,t)$.
  Take a nonnegative function $\zeta \in
  C_{0}^{\infty}$
  such that $\int_{\R^d}\zeta dx =1$.
  For arbitrary $\eps > 0$, we define the operator $S_{\eps}$ by
  \begin{equation}\label{eq:S_eps}
    S_{\eps}h(x) = \eps^{-d} \zeta\left(\frac{x}{\eps}\right) \ast h(x)\ \ {\rm for}\ x\in\mathbb{R}^d.
  \end{equation}
  Since the convolution is representable as an inner product, 
  from the definition of generalized solutions we have that
  \begin{eqnarray}\label{eq:d1}
    \begin{split}
      S_{\eps}u(t,x) = & S_{\eps}\vf(x) + \int_{t}^{T}
      S_{\eps}\big(a^{ij}u_{x^i x^j}+b^{i}u_{x^i}+cu
      +\sigma^{ik}q^{k}_{x^i}+\nu^{k}+f \big)(s,x) \,ds\\
      &-\int_{t}^{T} S_{\eps}q^{k}(s,x)\,dW_{s}^{k}
    \end{split}
  \end{eqnarray}
  holds on a $(t,x)$-independent set of the full measure.
  Then in view of $q \in L^{2}_{\sP}W^{m-1,2}$, by the Burkholder-Davis-Gundy
  (BDG) inequality and Sobolev's theorem, we have
  \[
  \begin{split}
    &\lim_{\eps\downarrow 0} \E \sup_{t \leq T} \bigg|
    \int_{t}^{T} \big[ S_{\eps}q^{k}(s,x) - q^{k}(s,x) \big]\,dW_{s}^{k}
    \bigg|^{2}\\
    & \leq C \lim_{\eps\downarrow 0} \E
    \int_{0}^{T} \big\| S_{\eps}q(s,\cdot) - q (s,\cdot)\big\|_{C^0}^{2}\,ds\\
    & \leq C \lim_{\eps\downarrow 0} \E \int_{0}^{T}
    \big\| S_{\eps}q(s) - q(s) \big\|^{2}_{m-2,2}\,ds, \quad \forall x\in \R^d.
  \end{split}
  \]

  In virtue of the properties of an averaging operator, the last limit
  vanishes. Thus, assertion (ii) follows immediately by taking $\eps\downarrow 0$ in
  \eqref{eq:d1}.
\end{proof}

\begin{rmk} Corollary
\ref{cor:001} relaxes the conditions of  Tang~\cite{Tang05}. Our proof of Corollary \ref{cor:001} (ii) is  direct.
\end{rmk}

Moreover, Theorem \ref{thm:main} allows us to prove the boundedness of the
first component of the solution and its derivatives, provided the coefficient $f$, the terminal data $\xi$ and their derivatives are bounded.

\begin{cor}
  Let conditions
  \emph{($\A_m$)}, \emph{($\DP$)}, and
  \emph{($\F_m$)} with $m\geq 1$ be satisfied,
  \[f \in L^{\infty}_{\sP}W^{m,\infty},~~{\rm and}~~
  \vf \in L^{\infty}(\Omega,\sF_{T},W^{m,\infty}).\]
  Then BSPDE~\eqref{eq:main}-\eqref{con:terminal} has a unique generalized
  solution $(u,q)$ satisfying \eqref{est:p=2}. Moreover, for arbitrary
  multi-index $\alpha$ s.t. $|\alpha|\leq m$, it holds that
  \[|D^{\alpha}u(\omega,t,x)| \leq C\]
  for almost all $(\omega,t,x)$.
\end{cor}

\begin{proof}
We only need to prove the last assertion. Set
\begin{center}$ N_p := (\E
\|\vf\|_{m,p}^{p})^{1/p} + (\E\int_{0}^{T} \|f(t)\|_{m,p}^{p}\,dt)^{1/p}.$
\end{center}
By interpolation, it holds that $f \in L^{p}_{\sP}W^{m,p}$ and $\vf \in
L^{p}(\Omega,\sF_{T},W^{m,p})$ for each $p\geq 2$, and
\[N_{p} \leq N^{2/p}_2 N^{1-2/p}_{\infty}\leq N_2 + N_\infty.\]
From Theorem \ref{thm:main} we have
\begin{eqnarray*}
\|D^{\alpha}u\|_{L^{\infty}(\Omega\times [0,T]\times\R^d)} \leq \sup_{p\geq
2}\bigg[\sup_{t\leq T} \Big(\E\|u(t)\|_{m,p}^{p}\Big)^{1/p}\bigg]\\
\leq \sup_{p\geq 2} C p^{2/p} N_{p} \leq C (N_2 + N_{\infty}).
\end{eqnarray*}
The proof is complete.
\end{proof}

\section{Proof of Theorem~\ref{thm:main}}

In this section, we  prove  Theorem~\ref{thm:main}. We begin with the following change of unknown variables in BSPDE~\eqref{eq:main}:
\[ r^{k} := q^{k} + \sigma^{ik}u_{x^i}.\]
Define
\begin{equation}\alpha^{ij} := \frac{1}{2}\sigma^{ik}\sigma^{jk} \quad \hbox{\rm and }
\quad \widetilde{b}^{i}:= b^{i} - \sigma^{ik}_{x^j} \sigma^{jk} - \nu^{k} \sigma^{ik}.
\end{equation} Then
BSPDE~\eqref{eq:main} can be rewritten into the following BSPDE:
\begin{eqnarray}\label{eq:trans}
  \begin{split}
    du = & - \big[ (a^{ij}-2\alpha^{ij}) u_{x^i x^j}
    + \widetilde{b}^{i} u_{x^i} + c u + \sigma^{ik} r^{k}_{x^i}
    + \nu^{k} r^{k} + f \big] \,dt \\
    & + ( r^{k} - \sigma^{ik}u_{x^i} ) \,dW^{k}_t,
  \end{split}
\end{eqnarray}
with the pair $(u,r)$ being the unknown.

Take $G(\cdot) \in C^{2}[0,\infty)$ such that $G(s),G'(s)>0$ and $G''(s) \geq
0$ for all $s\in (0,\infty)$. Applying It\^o's formula formally to compute
$G(\sum_{|\alpha| \leq m} |D^{\alpha} u|^{2})$, we immediately have
\begin{eqnarray}\label{eq:ito}
  \begin{split}
    d G\bigg(\sum_{|\alpha| \leq m}|D^{\alpha} u|^{2}& \bigg)
   =- \Theta(u,r,f; x,t,\omega) \,dt \\
    & +
    2 G'\bigg(\sum_{|\alpha| \leq m} |D^{\alpha} u|^{2} \bigg)
    \sum_{|\beta| \leq m}
    D^{\beta} u D^{\beta}( r^{k} - \sigma^{ik}u_{x^i} ) \,dW^{k}_t,
  \end{split}
\end{eqnarray}
where
\begin{eqnarray}\label{formula:Theta}
  \begin{split}
    & \Theta(u,r,f;x,t,\omega)\\
    & := 2G' \bigg( \sum_{|\alpha| \leq m} |D^{\alpha} u|^{2} \bigg)
    \sum_{|\beta| \leq m} D^{\beta}u  D^{\beta}
    \big[ (a^{ij}-2\alpha^{ij}) u_{x^i x^j}
    + \widetilde{b}^{i} u_{x^i} + c u \\
    & ~~~~ + \sigma^{ik} r^{k}_{x^i}
    + \nu^{k} r^{k} + f \big]
    - G' \bigg( \sum_{|\alpha| \leq m} |D^{\alpha} u|^{2} \bigg)
    \sum_{|\beta| \leq m} \|D^{\beta}
    ( r - \sigma^{i}u_{x^i})\|^{2}\\
    & ~~~~ - 2
    G''\bigg( \sum_{|\alpha| \leq m} |D^{\alpha} u|^{2} \bigg)
    \bigg\| \sum_{|\beta| \leq m} D^{\beta}u D^{\beta}
    (r-\sigma^{i}u_{x^i})\bigg\|^{2}.
  \end{split}
\end{eqnarray}

Note that the expression $\Theta(u,r,f;x,t,\omega)$ involves the leading coefficients  $a^{ij}$ and $\sigma^{ik}$ in our degenerate parabolic
equation. The following estimate turns out to be crucial in our arguments. In what follows, $\varepsilon$ is a generic constant which can be
chosen to be sufficiently small.

\begin{lem}\label{lem:basic}
  Let $m\geq 0,t\in [0,T]$ and $u \in W^{m+2,2},r\in W^{m+1,2}(d')$. For $u \in W^{m,2}$ and $r\in W^{m,2}(d')$, define
  \begin{eqnarray}
  \Psi := \sum_{|\alpha| \leq m} |D^{\alpha} u|^{2}
   \ and\ \Upsilon:=\sum_{|\alpha|\leq m}\|D^{\alpha} r\|^{2}.
   \end{eqnarray}
  Then under conditions \emph{($\A_m$)} and \emph{($\DP$)}, we have
  \begin{eqnarray}\label{est:basic}
  \begin{split}
    \int_{\bR^d} \Theta(u,r,f; x,t,\omega)\, dx
    \leq &-(1-\varepsilon)\int_{\bR^d}
    G'(\Psi)\Upsilon \,dx
    + {C\over\varepsilon} \int_{\R^d}\big[ G(\Psi) + G'(\Psi)\Psi \big] \,dx \\
    & + \sum_{|\beta| \leq m} \int_{\bR^d} G'(\Psi)|D^{\beta}f|^{2} \,dx,
  \end{split}
  \end{eqnarray}
provided that every integral in the right-hand side is finite.
\end{lem}

With the aid of Lemma \ref{lem:basic}, we are able to first establish, under the super-parabolic condition, the existence and uniqueness of
the $W^{m,p}$ solution ($p\geq2$) to BSPDEs with smooth coefficients. For this purpose, consider the following two assumptions.

\medskip
($\SP$)~ (\emph{super-parabolicity}) There is a constant $\delta > 0$ such that the matrix $$ 2a-\sigma \sigma^* \geq \delta I_{d}.$$

\medskip
(\B)~ The coefficients $a^{ij}, b^{i}, c, \sigma^{ik}, \nu^{k}$ are infinitely differentiable in $x$ for all $(\omega,t)$, and their
derivatives of orders up to $n$ are dominated by a constant depending on $n$. Moreover, we assume that
\[ f \in \bigcap_{m\in \mathbb{Z}_{+}}L^{2}_{\sP}W^{m,2},
~~~~\vf \in \bigcap_{m\in \mathbb{Z}_{+}} L^{2}(\Omega,\sF_{T},W^{m,2}).\] We
always choose an appropriate modification of $f$ such that for all $t$, $\E
\|f(t)\|_{m,2}^2 < \infty$ for all $m \in \mathbb{Z}_+$.

\medskip
We have the following lemma.

\begin{lem}\label{lem:sm.eq.est}
  Assume that $m \geq 0$, $p \geq 1$, $f \in L^{2p}_{\sP}W^{m,2p}$,
  and $\vf \in L^{2p}(\Omega,\sF_{T},W^{m,2p})$.
  Then under conditions \emph{($\A_m$)},
  \emph{($\SP$)}, and \emph{($\B$)},
  BSPDE~\eqref{eq:main}-\eqref{con:terminal} has
  a unique generalized solution $(u,q)$ such that
  \begin{equation}
    u \in L^{2p}_{\sP}CW^{m,2p}\quad \hbox{\rm and }\quad
    \Psi^{\frac{p-1}{2}}
    \Upsilon^{\frac{1}{2}}
    \in L^{2}_{\sP}W^{0,2},
  \end{equation}
  with
  \begin{equation}
\Psi := \sum_{|\alpha|\leq m}|D^{\alpha}u|^{2}\quad \hbox{ \rm and }\quad  \Upsilon:=\sum_{|\alpha|\leq m}\|D^{\alpha} r\|^{2}
\end{equation}
for $r := q+u_x\sigma$. Moreover, we have
  \begin{eqnarray}\label{est:sm.eq}
    \begin{split}
      \E \sup_{t\leq T}\int_{\R^d}\Psi^{p}\,dx
      + \E \int_{0}^{T}\!\!\!\int_{\R^d}
      \Psi^{p-1}\Upsilon \,dx~~~~~~~~\\
      \leq C{\rm e}^{Cp}\E\bigg(\|\vf\|^{2p}_{m,2p}
      + \int_{0}^{T}\|f(t)\|^{2p}_{m,2p} \,dt \bigg),
    \end{split}
  \end{eqnarray}
where the constant $C$ is independent of $\delta$ and $p$.
\end{lem}

\begin{rmk}\label{dtz5}
(i) The proofs of Lemmas \ref{lem:basic} and \ref{lem:sm.eq.est} are both quite technical and lengthy, and thus they are deferred to Sections
5 and 6, respectively.  We would like to stress here that they play a key role in the proof of Theorem~\ref{thm:main} and many calculation
techniques are developed in their own proofs. For example, we apply It\^o's formula in Lemma~\ref{lem:basic} to the expression
$$
  G\left(\sum_{|\alpha|\leq m}|D^{\alpha}u|^2\right)
$$
rather than to the  expression $G\left(|D^{\alpha}u|^2\right)$ for each multi-index $\alpha$ which is conventionally used in the study of
$W^{m,p}$-theory of degenerate SPDEs (c.f. \cite{KrRo82,Rozo90}).
It is difficult to derive a $W^{m,p}$-estimate ($p\geq2$) of solutions to BSPDEs by computing the expression $G\left(|D^{\alpha}u|^2\right)$
for each multi-index $\alpha$, even under the super-parabolicity condition, we have to deal with the new terms  $G'D^{\alpha}uD^{\beta}r$ with
$\alpha$ and $\beta$ being different multi-indices, and they turn out to be a trouble (see our more detailed exposition in the introduction).
For another example, $G$ is taken specifically for a localization method in the proof of Lemma~\ref{lem:sm.eq.est} when applying It$\hat {\rm
o}$'s formula to the solution of BSPDE. Though the stopping time arguments are conventionally used to get $W^{m,p}$-estimate of solutions to
SPDEs (c.f. \cite{KrRo77,Rozo90}), the localization method is conventionally used for BSPDEs and backward stochastic equations as well
(see~\cite{ZZ10}, where the localization method is used to get the $L^{p}(dx)$ solutions of backward doubly stochastic differential
equations).\smallskip

(ii) Treating the sum $q+\nabla u \sigma$ as a unity in our approach is a natural convenience of technical calculations, since it appears in
the diffusion term of a BSDE (see e.g., \cite{MaYo97,Tang05}).

(iii) By the equivalence of norms, \eqref{est:sm.eq} implies
  \begin{equation*}
      \E \sup_{t \leq T} \|u(t)\|_{m,2p}^{2p}
      \leq Ce^{Cp} \E \bigg( \|\vf\|_{m,2p}^{2p}
      + \int_{0}^{T} \|f(t)\|_{m,2p}^{2p}\,dt \bigg)
  \end{equation*}
with the same constant $C$ as in \eqref{est:sm.eq}.
\end{rmk}

Next we deal with the super-parabolic BSPDE with general coefficients.
\begin{lem}\label{lem:superp.eq}
  Assume that $m \geq 0$, $p \geq 2$ and conditions \emph{($\A_m$)},
  \emph{($\SP$)} and \emph{($\F_m$)} are satisfied. If
  $f \in L^{p}_{\sP}W^{m,p}$ and $\vf \in L^{p}(\Omega,\sF_{T},W^{m,p})$,
  then BSPDE~\eqref{eq:main}-\eqref{con:terminal} has
  a unique generalized solution $(u,q)$ satisfying
    $$u \in L^{2}_{\sP}W^{m+1,2} \cap L^{2}_{\sP}CW^{m,2}
    \cap L^{p}_{\sP}C_{w}W^{m,p}\ and\ q \in L^{2}_{\sP}W^{m,2}.$$
  Moreover,
    \begin{eqnarray}\label{est:superp.eq.2}
      \begin{split}
      \E \sup_{t \leq T} \|u(t)\|_{m,2}^{2}
      + \E\int_{0}^{T} \|(q+u_x\sigma)(t)\|_{m,2}^{2}\,dt~~~~~~\\
      \leq C \E \bigg( \|\vf\|_{m,2}^{2}
      + \int_{0}^{T} \|f(t)\|_{m,2}^{2}\,dt \bigg)
      \end{split}
    \end{eqnarray}
    and
    \begin{equation}\label{est:superp.eq.p}
      \E \sup_{t \leq T} \|u(t)\|_{m,p}^{p}
      \leq Ce^{Cp} \E \bigg( \|\vf\|_{m,p}^{p}
      + \int_{0}^{T} \|f(t)\|_{m,p}^{p}\,dt \bigg),
    \end{equation}
    where the constant $C$ is independent of
    $\delta$ and $p$.
\end{lem}

\begin{proof}
In view of Du and Meng~\cite[Theorem 2.3]{DuMe10} and Remark \ref{rmk:2-1}, BSPDE~\eqref{eq:main}-\eqref{con:terminal} has a unique
generalized solution $(u,q)$ satisfying
  $$u \in L^{2}_{\sP}W^{m+1,2} \cap L^{2}_{\sP}CW^{m,2},~~~~
  q \in L^{2}_{\sP}W^{m,2}.$$
Moreover,
\begin{eqnarray}\label{ieq:4.2-1}
  \begin{split}
      \E \sup_{t \leq T} \|u(t)\|_{m,2}^{2}
      + \E\int_{0}^{T} \Big[ \|u(t)\|_{m+1,2}^{2}
      + \|q(t)\|_{m,2}^{2}\Big]\,dt~~~~~~\\
      \leq C(\delta) \E \bigg( \|\vf\|_{m,2}^{2}
      + \int_{0}^{T} \|f(t)\|_{m-1,2}^{2}\,dt \bigg).
  \end{split}
\end{eqnarray}

Applying the averaging operator $S_{\eps}$ defined by \eqref{eq:S_eps} to the functions  $h = a^{ij},b^i,c,\sigma^{ik}$ and $\nu^k$, we set
\[
h_{n} := S_{\frac{1}{n}}h, \quad f_{n} := S_{\frac{1}{n}}f, \quad \vf_{n}: = S_{\frac{1}{n}}\vf.
\]
In view of the properties of the averaging operator $S_{\eps}$, we have the
following assertions:
\begin{enumerate}
  \item[(i)]
  the functions $a_{n}^{ij},b_{n}^i,c_{n},\sigma_{n}^{ik}$ and $\nu_{n}^k$
  satisfy the conditions of Lemma \ref{lem:sm.eq.est}
  for all $n$ with the same constants $K_m$ and $\delta$;
  \item[(ii)]
  for $h = a^{ij},b^i,c,\sigma^{ik}$, $\nu^k$ and arbitrary
  multi-index $\alpha$, we have
  $$ D^{\alpha}h_{n}(\omega,t,x) \rightarrow D^{\alpha}h(\omega,t,x),
  ~~~~\textrm{as}~ n \rightarrow \infty$$
  uniformly w.r.t. $(\omega,t,x)$;
  \item[(iii)]
  for each $n$, both functions $f_{n}$ and $\vf_{n}$ satisfy
  the conditions of Lemma \ref{lem:sm.eq.est}, and for $p'=2$ or $p$, we have
  \begin{equation}
    \lim_{n\rightarrow \infty} \E \bigg[ \|\vf_{n}-\vf\|_{m,p'}^{p'}
    + \int_{0}^{T} \|f_{n}(t)-f(t)\|_{m,p'}^{p'}\,dt \bigg] = 0.
  \end{equation}
\end{enumerate}

Consider the following BSPDE
\begin{numcases}{}\label{eq:4.2-2}
    du_{n}=-(a_{n}^{ij}u_{n,x^i x^j}+b_{n}^{i}u_{n,x^i}
    +c_{n}u_{n}+\sigma_{n}^{ik}q^{k}_{n,x^i}+\nu_{n}^k q_{n}^k +
    f_{n})\,dt + q_{n}^k \,dW^k_t,\nonumber\\
    u_{n}(T) = \vf_{n}.
\end{numcases}
The existence and uniqueness of the solution of (\ref{eq:4.2-2}) is indicated by Lemma \ref{lem:sm.eq.est}. Moreover, in view of  Lemma
\ref{lem:sm.eq.est} (for $p=1$), we have
\begin{eqnarray}\label{ieq:4.2-3}
    \begin{split}
      \E \sup_{t \leq T} \|u_{n}(t)\|_{m,2}^{2}
      + \E\int_{0}^{T} \|q_{n}(t)+\sigma_{n}^{i}u_{n,x^i}(t)\|_{m,2}^{2}dt~~~~~~\\
      \leq C \E \bigg( \|\vf_{n}\|_{m,2}^{2}
      + \int_{0}^{T} \|f_{n}(t)\|_{m,2}^{2}\,dt \bigg),
    \end{split}
\end{eqnarray}
where the constant $C$ is independent of 
$\delta$ and $n$. To see this, set
\begin{gather*}
  \widetilde{u}_{n} := u_{n}-u,\quad \widetilde{q}_{n}: = q_{n}-q,
  \quad \widetilde{\vf}_{n} := \vf_{n}-\vf, \\
  \widetilde{f}_{n} := (a_{n}^{ij}-a^{ij})u_{x^ix^j} + (b_{n}^{i}-b^{i})u_{x^i}
  + (c_{n}-c)u \\
  + (\sigma_{n}^{ik}-\sigma^{ik})q^{k}_{x^i}
  + (\nu_{n}^{k}-\nu^{k})q^{k} + f_{n}-f,
\end{gather*}
then clearly $\widetilde{f}_{n} \in L^{2}_{\sP}W^{m-1,2}$. By the dominated convergence theorem, we have
\begin{equation*}
  \lim_{n\rightarrow \infty} \E \bigg[ \|\widetilde{\vf}_{n}\|_{m,2}^{2}
  + \int_{0}^{T} \|\widetilde{f}_{n}(t)\|_{m-1,2}^{2}\,dt \bigg] = 0.
\end{equation*}
Noting that $(\widetilde{u}_{n},\widetilde{q}_{n})$ is a generalized solution
of BSPDE
\begin{numcases}{}
    d\widetilde{u}_{n}=-(a_{n}^{ij}\widetilde{u}_{n,x^i x^j}
    +b_{n}^{i}\widetilde{u}_{n,x^i}
    +c_{n}\widetilde{u}_{n}+\sigma_{n}^{ik}\widetilde{q}^{k}_{n,x^i}
    +\nu_{n}^k \widetilde{q}_{n}^k +
    \widetilde{f}_{n})\,dt + \widetilde{q}_{n}^k \,dW^k_t\nonumber\\
    \widetilde{u}_{n}(T)=\widetilde{\vf}_{n}\nonumber
\end{numcases}
and applying
a similar estimate as \eqref{ieq:4.2-1}, we get
\begin{equation*}
  \lim_{n\rightarrow \infty} \E \bigg\{ \sup_{t \leq T}
  \|\widetilde{u}_{n}(t)\|_{m,2}^{2}
  + \int_{0}^{T} \Big[ \|\widetilde{u}_{n}(t)\|_{m+1,2}^{2}
  + \|\widetilde{q}_{n}(t)\|_{m,2}^{2}\Big]\,dt \bigg\} = 0.
\end{equation*}
Passing $n\rightarrow \infty$ in \eqref{ieq:4.2-3},
\eqref{est:superp.eq.2} follows.

Next we need to prove that $u \in L^{p}_{\sP}C_{w}W^{m,p}$ and to deduce \eqref{est:superp.eq.p}. We have known that $u_{n},u \in
L^{2}_{\sP}CW^{m,2}$ and
  $\lim_{n\rightarrow \infty} E \sup_{t\leq T} \| u_{n}(t) - u(t)
  \|_{m,2}^{2} = 0$.
Hence for arbitrary countable set $\{t_i\}\subset[0,T]$, $\eta_j \in
C^{\infty}_{0}$ and multi-index $\alpha$ with $|\alpha| \leq m$, we have
\begin{eqnarray}\label{ieq:4.2-4}
    \la D^{\alpha}u(t_i),\eta_{j} \ra_{0}
    = \lim_{n'\rrow \infty} \la D^{\alpha}u_{n'}(t_i),\eta_{j} \ra_{0}
    \leq \lim_{n'\rrow \infty} \|u_{n'}(t_i)\|_{m,p}
    \|\eta_{j}\|_{\frac{p}{p-1}}
\end{eqnarray}
on a full-measure set, where $\{n'\}$ is an appropriate subsequence of $\{n\}$.

Since $u \in L^{2}_{\sP}CW^{m,2}$, it is easy to see that the left-hand side
of \eqref{ieq:4.2-4} is continuous at $t_{i}$ for each $\eta_j$ on a
full-measure set. Let $t_{i}$ run through the rational points of $[0,T]$ and
$\eta_j $ run through a dense subset of the unit sphere of
$L^{\frac{p}{p-1}}(\R^d)$. Then by \eqref{ieq:4.2-4} we obtain that for
arbitrary $t\in[0,T]$ and multi-index $\alpha$ with $|\alpha| \leq m$,
\begin{eqnarray*}
  \begin{split}
    &\|D^{\alpha}u(t)\|_{0,p}
    = \sup_{j}\la D^{\alpha}u(t),\eta_{j} \ra_{0}\\
    &\leq \sup_{j}\sup_{i}\la D^{\alpha}u(t_i),\eta_{j} \ra_{0}
    \leq \varliminf_{n'\rrow \infty}\sup_{t\leq T}\|u_{n'}(t)\|_{m,p}\ \ \ a.s.
  \end{split}
\end{eqnarray*}
On the other hand, by Lemma \ref{lem:sm.eq.est} we know that
\begin{equation}\label{ieq:4.2-5}
  \E \sup_{t \leq T} \|u_{n'}(t)\|_{m,p}^{p}
  \leq C{\rm e}^{Cp}\E \bigg( \|\vf_{n'}\|_{m,p}^{p}
  + \int_{0}^{T} \|f_{n'}(t)\|_{m,p}^{p}\,dt \bigg).
\end{equation}
Hence, in view of assertion (iii) in this proof and Fatou's lemma, we obtain
\eqref{est:superp.eq.p}. This estimate, along with the continuity of the
left-hand side of \eqref{ieq:4.2-4} at $t_i$ on a full-measure set and the
separability of $L^{\frac{p}{p-1}}(\R^d)$, yields that $u(t)\in W^{m,p}$ for
all $t$ and $u(t)$ is weakly continuous in $W^{m,p}$ w.r.t. $t$ on a
full-measure set. The proof is complete.
\end{proof}

\begin{rmk}
It is worth noting that  the $W^{m,p}$ estimate for $p\geq2$ of Du, Qiu, and Tang~\cite{DQT10} requires the super-parabolicity condition
($\SP$), and  the generic constant $C$ depends on $\delta$ of Condition ($\SP$).  Our estimates~(\ref{est:superp.eq.2}) and
(\ref{est:superp.eq.p}) do not require such a dependence, but at a cost of losing the $W^{m+1,p}$ estimate.
\end{rmk}

We are now in a position to prove Theorem~\ref{thm:main}.

\begin{proof}[Proof of Theorem \ref{thm:main}]
\emph{Uniqueness.} We need to prove that if $f=0$, $\vf=0$ and $(u,q)$ is a generalized solution of
BSPDE~\eqref{eq:main}-\eqref{con:terminal}, then $u=0$ and $q=0$ a.e. $(\omega,t)\in\Omega\times[0,T]$. Recalling Remark \ref{rmk:2-1}, we
always choose the continuous modification of $u$ and apply It\^o's formula to $e^{Lt}\|u(t)\|_{0,2}^{2}$. Using Lemma \ref{lem:basic} in the
case that $G(s)=s^2$ and $m=0$, we can easily get
  $$ e^{Lt}\|u(t)\|_{0,2}^{2}
  \leq (C-L)\int_{t}^{T} e^{Ls}\|u(s)\|_{0,2}^{2}\,ds + M_{T}-M_t,$$
  where $M_{t}, t\in [0,T],$ is a continuous martingale with $M_0 = 0$ and the
  constant $C$ comes from Lemma \ref{lem:basic}. By taking $L=C$, we have
  $0 \leq e^{Ct}\|u(t)\|_{0,2}^{2} \leq M_{T}-M_t$, which implies that
  $u \equiv 0$ and then $q \equiv 0$. The uniqueness is proved.
\medskip

\noindent\emph{Existence.} Take $\eps > 0$ and consider the following
BSPDE
\begin{numcases}{}\label{eq:eps}
  du^{\eps}=-(\eps \Delta u^{\eps} + a^{ij}u^{\eps}_{x^i x^j}+b^{i}u^{\eps}_{x^i}
  +cu^{\eps}+\sigma^{ik}q^{\eps,k}_{x^i}+\nu^k q^{\eps,k} +
  f)\,dt + q^{\eps,k} \,dW^k_t,\nonumber\\
  u^{\eps}(T)=\vf.
\end{numcases}
It is clear that BSPDE \eqref{eq:eps} satisfies the conditions of Lemma
\ref{lem:superp.eq}, thus \eqref{eq:eps} has a unique generalized solution
$(u^{\eps},q^{\eps})$ satisfying all the assertions in Lemma
\ref{lem:superp.eq}. Set $r^{\eps} = q^{\eps} + \sigma^{i}u^{\eps}_{x^i}$,
then we have
\begin{eqnarray*}
  \E \sup_{t \leq T} \|u^{\eps}(t)\|_{m,2}^{2}
  + \E\int_{0}^{T} \|r^{\eps}(t)\|_{m,2}^{2}\,dt
  \leq C \E \bigg( \|\vf\|_{m,2}^{2}
  + \int_{0}^{T} \|f(t)\|_{m,2}^{2}\,dt \bigg),
\end{eqnarray*}
where the constant $C$ is independent of $\eps$.

Therefore, we can find a sequence $\eps_n \downarrow 0$ and
$$(u,r) \in L^{2}_{\sP}W^{m,2} \times L^{2}_{\sP}W^{m,2}(d')$$
such that $(u^{\eps_n},r^{\eps_n})$ converges weakly in $L^{2}_{\sP}W^{m,2}
\times L^{2}_{\sP}W^{m,2}(d')$ to $(u,r)$ as $n \rrow \infty$. By the resonance
theorem, we have
\begin{eqnarray}\label{ieq:5-1}
  \E\int_{0}^{T} \|r(t)\|_{m,2}^{2}\,dt
  \leq C \E \bigg( \|\vf\|_{m,2}^{2}
  + \int_{0}^{T} \|f(t)\|_{m,2}^{2}\,dt \bigg).
\end{eqnarray}

If setting $q = r - u_x\sigma$, we know that $q^{\eps_n}$ converges weakly in $L^{2}_{\sP}W^{m-1,2}(d')$ to $q$. Therefore, it is obvious that
for any $\eta \in C^{\infty}_{0}$, all terms on both sides of \eqref{eq:weaksol} with $(u,q)$ replaced by $(u^{\eps_n},q^{\eps_n})$, converge
weakly in $L^{2}_{\sP}(\Omega\times[0,T])$ to the corresponding terms for $(u,q)$ as $\eps_n \downarrow 0$. Since the operators of integration
and stochastic integration are continuous in $L^{2}_{\sP}(\Omega\times[0,T])$, they are weakly continuous and thus $(u,q)$ is a generalized
solution of \eqref{eq:main}-\eqref{con:terminal}. Furthermore, noting Remark 2.1, we can always take $u \in L^{2}_{\sP}CW^{m-1,2}$.
\medskip

Since $u^{\eps_n}$ converges weakly in $L^{2}_{\sP}W^{m,2}$ to $u$ as $n \rrow \infty$, by Banach-Saks theorem we construct a sequence $u^k$
from finite convex combinations of $u^{\eps_n}$ such that $u^k$ converges strongly to $u$ in $W^{m,2}$ for a.e. $t\in[0,T]$ a.s. Hence for
some countably dense set $\{t_i\}\subset[0,T]$, any $\eta_j \in C^{\infty}_{0}$ and $|\alpha| \leq m$, we have
\begin{eqnarray}\label{ieq:5-2}
  \begin{split}
    (-1)^{|\alpha|}\la u(t_i),D^{\alpha}\eta_{j} \ra_{0}
    =\la D^{\alpha}u(t_i),\eta_{j} \ra_{0}
    = \lim_{k\rrow \infty} \la D^{\alpha}u^k(t_i),\eta_{j} \ra_{0}\\
    \leq \lim_{k\rrow \infty} \|u^k(t_i)\|_{m,p}
    \|\eta_{j}\|_{\frac{p}{p-1}}
  \end{split}
\end{eqnarray}
on a full-measure set. Since $u \in L^{2}_{\sP}CW^{m-1,2}$, we know that for
any $\eta_j$, the first term of \eqref{ieq:5-2} is continuous at $t_{i}$ on a
full-measure set. Allow $\eta_j \in C^{\infty}_{0}$ to run through a dense
subset of the unit sphere of $L^{\frac{p}{p-1}}(\R^d)$, by \eqref{ieq:5-2} we
obtain that for $|\alpha| \leq m$,
\begin{eqnarray*}
  \begin{split}
    &\|D^{\alpha}u(t)\|_{0,p}
    = \sup_{j}\la D^{\alpha}u(t),\eta_{j} \ra_{0}\\
    &\leq \sup_{j}\sup_{i}\la D^{\alpha}u(t_i),\eta_{j} \ra_{0}
    \leq \varliminf_{k\rrow \infty}\sup_{t\leq T}\|u^k(t)\|_{m,p}\ \ \ a.s.
  \end{split}
\end{eqnarray*}
and for all $t$,
  \[\sup_{t\leq T}\|D^{\alpha}u(t)\|_{0,p}
  \leq \varliminf_{k\rrow \infty}\sup_{t\leq T}\|u^k(t)\|_{m,p}\ \ \ a.s.\]
On the other hand, noting that $u^{\eps_n}$ satisfies
\eqref{est:superp.eq.p}, by Jensen's inequality we have
\begin{equation*}
  \E \sup_{t \leq T} \|u^k(t)\|_{m,p}^{p}
  \leq C e^{Cp} \E\bigg( \|\vf\|_{m,p}^{p}
  + \int_{0}^{T} \|f(t)\|_{m,p}^{p}\,dt \bigg).
\end{equation*}
Therefore, \eqref{est:p=p} follows from Fatou's lemma. Similarly, we can establish the same estimate for $p=2$ which, along with
\eqref{ieq:5-1}, yields \eqref{est:p=2}. Using a similar argument as in the end of the proof of Lemma \ref{lem:superp.eq}, we can deduce $u
\in L^{2}_{\sP}C_{w}W^{m,2} \cap L^{p}_{\sP}C_{w}W^{m,p}$. The proof is complete.
\end{proof}

%
%
%

\section{Application: maximum principle for optimal control
of degenerate SPDEs}

One  main application of BSPDEs is to formulate the maximum principle for optimal controls of SPDEs (c.f. \cite{Bens83,Tang98a,Zhou93}). Since
Theorem~\ref{thm:main} does not require the super-parabolic condition, we give an example to illustrate the application in stochastic control
theory.\medskip

Let $\Gamma$ be a non-empty Borel set in some Euclidean space. For $x\in\R^d$
and $\phi\in C^\infty_0$, we define two differential operators
$\{\mathcal{L}(t,v):t\in[0,T],v\in\Gamma\}$ and
$\{\mathcal{M}^k(t,v):t\in[0,T],v\in\Gamma,k=1,2,\cdot\cdot\cdot,d'\}$ as follows:
\begin{eqnarray}
  \begin{split}
    \mathcal{L}(t,v)\phi(x)&:=[a^{ij}(t,x,v)\phi_{x^j}(x)]_{x^i}
    +b^i(t,x,v)\phi_{x^i}(x)+c(t,x,v)\phi(x),\\
    \mathcal{M}^k(t,v)\phi(x)&:=\sigma^{ik}(t,x,v)\phi_{x^i}(x)+v^k(t,x,v)\phi(x),
  \end{split}
\end{eqnarray}
where $a^{ij},b^i,c,\sigma^{ik}$ and $v^k$ are given real valued functions
defined in $[0,T]\times\R^d\times\Gamma$, $i,j=1,2,\cdot\cdot\cdot,d$ and
$k=1,2,\cdot\cdot\cdot,d'$. In the mean time, we write down the adjoint
operators of $\mathcal{L}(t,v)$ and $\mathcal{M}^k(t,v)$:
\begin{eqnarray}
  \begin{split}
    \mathcal{L}^*(t,v)\phi(x)&:=[a^{ij}(t,x,v)\phi_{x^j}(x)]_{x^i}
    -[b^i(t,x,v)\phi(x)]_{x^i}+c(t,x,v)\phi(x)\\
    {\cM^k}^*(t,v)\phi(x)&:=-[\sigma^{ik}(t,x,v)\phi(x)]_{x^i}+v^k(t,x,v) \phi(x).
  \end{split}
\end{eqnarray}
We denote by $\mathscr{V}_{ad}$ the totality of admissible controls which are $\Gamma$-valued, $\mathscr{F}$-adapted processes $\{V(t), 0\leq
t\leq T\}$.
\medskip

\noindent \textbf{Problem.} Given $V(\cdot)\in \mathscr{V}_{ad}$, we consider the following controlled linear SPDE (with the sate variable $x$
being omitted):
\begin{equation}\label{dtz1}
  \left\{\begin{array}{rcl}
    d\xi(t)&=&[\mathcal{L}(t,V(t))\xi(t)+F(t,V(t))]\,dt\\
    &&~+[{\cM^k}(t,V(t))\xi(t)
    +G^k(t,V(t))]\,dW_t^k,\ \ \ \ t\in[0,T];\\
    \xi(0)&=&\xi_0.
  \end{array}\right.
\end{equation}
The optimal control problem is to find $V(\cdot)\in \mathscr{V}_{ad}$ which
minimizes the cost functional below:
\begin{eqnarray}\label{dtz2}
J(V):=\E\bigg\{\int_0^T\langle f(t,V(t)),\xi^V(t)\rangle_0\,dt+\langle
\varphi,\xi^V(T)\rangle_0\bigg\},
\end{eqnarray}
where $f:[0,T]\times\Gamma\rrow W^{1,2}$ and $\varphi\in W^{1,2}$ are given.
\medskip

A process $\xi=\xi^V\in L^2_\mathscr{P}W^{1,2}$ is called a (generalized)
solution of SPDE \eqref{dtz1} for the control $V$ if, for each $\eta\in
C_0^\infty$ and a.e. $t\in[0,T]$,
\begin{eqnarray*}
&&\langle\xi(t),\eta\rangle_0=\langle\xi_0,\eta\rangle_0
+\int_0^t\langle\mathcal{L}(s,V(s))\xi(s)+F(s,V(s)),\eta\rangle_0\,ds\nonumber\\
&&\ \ \ \ \ \ \ \ \ \ \ \ \ \ \ \ +\int_0^t\langle{\cM^k}(s,V(s))\xi(s)+G^k(s,V(s)),\eta\rangle_0\,dW_s^k, \quad \bP\textrm{-a.s.}
\end{eqnarray*}

For this problem, we make the following hypotheses.

$\bullet$~ The functions
$a^{ij},b^i,c,\sigma^{ik},\nu^k:[0,T]\times\R^d\times\Gamma\rrow \R$ are
measurable in $(t,x,v)$ and continuous in $v$; $a^{ij},\sigma^{ik}$ and their
derivatives w.r.t. $x$ up to second order, as well as $b^i,c,\nu^k$ and their
first order derivatives, do not exceed a constant $K_1$ in absolute value.

$\bullet$~ (\emph{parabolicity}) For each $(t,x,v)\in
[0,T]\times\bR^d\times\Gamma$, the matrix
$$(2 a^{ij}-\sigma^{ik}\sigma^{jk})_{d \times d} \geq 0.$$

$\bullet$~ The function $f(t,v)\in W^{1,2}$ for each $(t,v)$, the function
$\vf\in W^{1,2}$ and $$\|f(t,v)\|_{1,2}+\|\varphi\|_{1,2}\leq K_1.$$

$\bullet$~ The functions $F,G^k:[0,T]\times\R^d\times\Gamma\rrow\R$ are
measurable in $(t,x,v)$ and continuous in $v$, furthermore, $F(t,v)\in
W^{1,2}$ and $G^k(t,v)\in W^{2,2}$ for each $(t,v)$, and
$$|F(t,x,v)|+|G^k(t,x,v)|+\|F(t,\cdot,v)\|_{1,2}
+\|G^k(t,\cdot,v)\|_{2,2}\leq K_1.$$

$\bullet$~ $\xi_0\in W^{1,2}$.
\medskip

Under above conditions, we know immediately from Krylov and Rozowskii
\cite{KrRo82} that the state equation \eqref{dtz1} has a unique solution
$\xi^{V}\in L^{2}_{\sP}C_{w}W^{1,2}$ for any given $V\in
\mathscr{V}_{ad}$.\medskip

The adjoint equation of SPDE \eqref{dtz1} reads
\begin{equation}\label{dtz3}
  \left\{\begin{array}{l}
    du=-\big[\mathcal{L}^*(t,V)u+{\cM^k}^*(t,V)q^k+f(t,V) \big]\,dt
    + q^k \,dW^k_t,\\
    u(T)=\vf.
  \end{array}\right.
\end{equation}
By Theorem \ref{thm:main}, for any given $V\in \mathscr{V}_{ad}$, BSPDE
\eqref{dtz3} has a unique solution $(u,q)\in L_{\sP}^2C_{w}W^{1,2}\times
L_{\sP}^2W^{0,2}(d')$.
\medskip

Now we can give the necessary condition of an optimal control for the general
SPDE system (\ref{dtz1}) with the cost functional (\ref{dtz2}).
\begin{prop}\label{dtz4}
Under above hypotheses, we assume that $\widetilde{V}$ is an optimal
control along with the corresponding optimal state $\widetilde{\xi}$. Then
for a.e. $t\in[0,T]$, we have the maximum condition
$$
H(t,\widetilde{\xi}(t),\widetilde{V}(t),u(t),q(t))
=\max_{v\in\Gamma}H(t,\widetilde{\xi}(t),v,u(t),q(t)),\ \ \
\bP\textrm{-a.s.},$$ where $(u,q)$ is the solution of (\ref{dtz3}) with
$V(t)=\widetilde{V}(t)$ and the Hamiltonian function $H$ is defined by
\begin{eqnarray*}
H(t,\phi,v,\zeta,\eta)&:=&-\langle\mathcal{L}(t,v)\phi,\zeta\rangle_0-\langle
F(t,v),\zeta\rangle_0\\
&&-\langle \cM^k(t,v)\phi,\eta^k\rangle_0 - \langle
G^k(t,v),\eta^k\rangle_0-\langle f(t,v),\phi\rangle_0
\end{eqnarray*}
for $(t,\phi,v,\zeta,\eta)\in[0,T]\times W^{1,2}\times\Gamma\times
W^{1,2}\times L^2(d')$.
\end{prop}
The procedure to the proof of Proposition (\ref{dtz4}) is very similar to Zhou~\cite[Theorem 5.1]{Zhou93}. Although the adjoint equation BSPDE
(\ref{dtz3}) only satisfies the degenerate parabolic condition rather than super-parabolic condition, Theorem~\ref{thm:main} guarantees that
degenerate parabolic condition still works.

It is well known that the optimal control problem of partially observed diffusions with general nonlinear cost functionals can be transformed
into an optimal control with complete observation of Zakai's equations with linear cost functionals. Therefore the previous results also
enable us to discuss the partially observed diffusions under degenerate parabolic condition.

\section{Proof of Lemma \ref{lem:basic}}

Define
\begin{equation}
  I  :=
  \int_{\bR^d} G'(\Psi)\Upsilon
  \,dx \quad \hbox{ \rm and }\quad
  J  := \int_{\bR^d}\big[ G(\Psi) + G'(\Psi)\Psi \big] \,dx.
\end{equation}
 In what follows, the notations $G(\Psi),G'(\Psi)$ and
$G''(\Psi)$ will be occasionally simplified as $G,G'$ and $G''$ in the following arguments, when no confusion occurs.
\medskip

Now we estimate all terms in (\ref{formula:Theta}). First of all, we
have
\begin{equation}\label{ieq:c1}
    2\sum_{|\beta| \leq m} \int_{\bR^d} G'
    D^{\beta}u D^{\beta}(cu) \,dx
    \leq C \int_{\bR^d} G' \Psi \,dx
    \leq C J.
\end{equation}
Then by Young's inequality, we have
\begin{eqnarray}\label{ieq:c2}
    \nonumber 2 \sum_{|\beta| \leq m} \int_{\bR^d} G'
    D^{\beta}u D^{\beta}f \,dx
    &\leq&  \sum_{|\beta| \leq m}
    \int_{\bR^d} \big(|D^{\beta} u|^{2}
    +|D^{\beta}f|^{2}\big) G' \,dx\\
    &\leq& J+\sum_{|\beta| \leq m} \int_{\bR^d} G'|D^{\beta}f|^{2} \,dx.
\end{eqnarray}
By the integration by parts, it is not hard to get
\begin{eqnarray}\label{ieq:c3}
    \nonumber 2 \sum_{|\beta| \leq m} \int_{\bR^d} G'
    D^{\beta}u D^{\beta}(\widetilde{b}^{i} u_{x^i}) \,dx
    \leq 2 \sum_{|\beta| \leq m} \int_{\bR^d}  \widetilde{b}^{i} G'
    D^{\beta}u D^{\beta}u_{x^i} \,dx
    + C J \\
    \leq \int_{\bR^d} \widetilde{b}^{i}G_{x^i}\,dx + C J
    \leq -\int_{\bR^d} \widetilde{b}^{i}_{x^i}G \,dx + C J
    \leq C J .~~~~
\end{eqnarray}
\medskip
Similarly as \eqref{ieq:c3}, it follows
\[
  \begin{split}
    - 2 \sum_{|\beta| \leq m} \int_{\bR^d} G'
    (a^{ij}-2\alpha^{ij})_{x^j}(D^{\beta}u) D^{\beta}
    u_{x^i} \,dx
    \leq C J .
  \end{split}
\]
Then we introduce a new notation $\sum_{\beta} ~:= \sum_{\beta_1 +
\beta_2 = \beta, |\beta_1|=1}$ to deal with the second order derivatives
of $u$ and we have
\begin{eqnarray}\label{ieq:c4}
  \begin{split}
    & 2 \sum_{|\beta| \leq m} \int_{\bR^d} G'
    D^{\beta}u D^{\beta} \big[ (a^{ij}-2\alpha^{ij})
    u_{x^i x^j}\big] \,dx\\
    &~~~~ \leq C J
    + 2  \sum_{|\beta| \leq m}\int_{\bR^d} G'
    (a^{ij}-2\alpha^{ij}) D^{\beta}u D^{\beta}
    u_{x^i x^j} \,dx\\
    &~~~~~~~~ + 2 \sum_{|\beta| \leq m} \sum_{\beta}
    \int_{\bR^d} G' D^{\beta}u D^{\beta_1}(a^{ij}-2\alpha^{ij})
    D^{\beta_2} u_{x^i x^j} \,dx \\
    &~~~~ \leq C J
    - 2  \int_{\bR^d} G'
    (a^{ij}-2\alpha^{ij}) D^{\beta}u_{x^i} D^{\beta}
    u_{x^j} \,dx\\
    &~~~~~~~~ -4 \int_{\bR^d} G'' (a^{ij}-2\alpha^{ij})
    \bigg( \sum_{|\beta| \leq m} D^{\beta}u D^{\beta}u_{x^i} \bigg)
    \bigg( \sum_{|\gamma| \leq m} D^{\gamma}u D^{\gamma}u_{x^j} \bigg)\,dx \\
    &~~~~~~~~ + 2 \sum_{|\beta| \leq m} \sum_{\beta}
    \int_{\bR^d} G' D^{\beta}u D^{\beta_1}(a^{ij}-2\alpha^{ij})
    D^{\beta_2} u_{x^i x^j} \,dx.
  \end{split}
\end{eqnarray}
\medskip

For the remaining four terms containing the function
$r$, first by
Cauchy-Schwarz inequality, we have 
\begin{eqnarray}\label{ieq:c5}
\begin{split}
    R_1 & := 2 \sum_{|\beta| \leq m} \int_{\bR^d} G'
    D^{\beta}u D^{\beta}(\nu^{k} r^{k})\,dx \leq C\sum_{|\beta| \leq m}
    \int_{\bR^d} G' \Psi^{\frac{1}{2}} \|D^{\beta}r \| \,dx\\
    & \leq {\eps\over3} I
    + {C\over{\varepsilon}} J.
\end{split}
\end{eqnarray}
By the integration by parts, it follows 
\begin{eqnarray}\label{ieq:c6}
  \begin{split}
    R_2 & := 2 \sum_{|\beta| \leq m} \int_{\bR^d} G'
    D^{\beta}u D^{\beta}(\sigma^{ik} r^{k}_{x^i})\,dx\\
    & \leq 2 \sum_{|\beta| \leq m} \int_{\bR^d} G'
    D^{\beta}u (\sigma^{ik} D^{\beta} r^{k})_{x^i} \,dx
    + {\eps\over3} I
    + {C\over\varepsilon} J \\
    & = - 4 \int_{\bR^d} G''
    \bigg[\sum_{|\beta| \leq m} (D^{\beta} u)\sigma^{ik} D^{\beta}u_{x^i}
    \bigg]
    \bigg(\sum_{|\gamma| \leq m} D^{\gamma}u D^{\gamma} r^{k}
    \bigg) \,dx\\
    & ~~~~- 2 \sum_{|\beta| \leq m}\int_{\bR^d} G'
    \sigma^{ik} D^{\beta}u_{x^i} D^{\beta} r^{k} \,dx
    + {\eps\over3}I
    + {C\over\varepsilon} J.
  \end{split}
\end{eqnarray}
Note that
$D^{\beta}(\sigma^{ik}u_{x^i})-\sigma^{ik}D^{\beta}u_{x^i}$ does not
contain the $m+1$ order derivatives of $u$, hence using the
integration by parts again
we have 
\begin{eqnarray}\label{ieq:c7}
  \begin{split}
    R_3 & := - \sum_{|\beta| \leq m}
    \int_{\bR^d} G' \| D^{\beta} (r - \sigma^{i} u_{x^i}) \|^{2}\,dx\\
    & = - \sum_{|\beta| \leq m}
    \int_{\bR^d} G' \big\| D^{\beta} r
    - \sigma^{i}D^{\beta}u_{x^i}
    - \big[D^{\beta}(\sigma^{i}u_{x^i})-\sigma^{i}D^{\beta}u_{x^i}\big]
    \big\|^{2}\,dx\\
    & \leq - I
    - \sum_{|\beta| \leq m} \int_{\bR^d} G'
    \|\sigma^{i}D^{\beta}u_{x^i}\|^{2}\,dx
    + 2 \sum_{|\beta| \leq m} \int_{\bR^d} G'
    \sigma^{ik} D^{\beta} u_{x^i} D^{\beta} r^{k} \,dx\\
    & ~~~~ - 2 \sum_{|\beta| \leq m} \int_{\bR^d} G'
    \sigma^{jk}D^{\beta}u_{x^j}
    \big[D^{\beta}(\sigma^{ik}u_{x^i})-\sigma^{ik}D^{\beta}u_{x^i}\big]
    \,dx + {\eps\over3} I  + {C\over\varepsilon} J\\
    & \leq - I
    - \sum_{|\beta| \leq m} \int_{\bR^d} G'
    \|\sigma^{i}D^{\beta}u_{x^i}\|^{2}\,dx
    + 2 \sum_{|\beta| \leq m} \int_{\bR^d} G'
    \sigma^{ik} D^{\beta} u_{x^i} D^{\beta} r^{k} \,dx\\
    & ~~~~ + 4 \int_{\bR^d} G''
    \bigg\{ \sum_{|\beta| \leq m}D^{\beta}u
    \big[D^{\beta}(\sigma^{ik}u_{x^i})-\sigma^{ik}D^{\beta}u_{x^i}\big]\bigg\}
    \bigg( \sum_{|\gamma| \leq m}\sigma^{jk}D^{\gamma}u D^{\gamma}u_{x^j}
    \bigg)\,dx\\
    & ~~~~ + 2 \sum_{|\beta| \leq m} \sum_{\beta} \int_{\bR^d} G'
    \sigma^{jk}D^{\beta}u D^{\beta_1}
    \sigma^{ik} D^{\beta_2}u_{x^i x^j}\,dx
    + {\eps\over3} I  + {C\over\varepsilon} J.
  \end{split}
\end{eqnarray}
The last term in (\ref{formula:Theta}) can be handled as follows: 
\begin{eqnarray}\label{ieq:c8}
  \begin{split}
    R_4 & := -2
    \int_{\bR^d} G''
    \bigg\| \sum_{|\beta| \leq m} D^{\beta}u D^{\beta}
    (r-\sigma^{i}u_{x^i})\bigg\|^{2}\,dx\\
    & = -2
    \int_{\bR^d} G''
    \bigg\| \sum_{|\beta| \leq m} D^{\beta}u \bigg(
    \Big\{ D^{\beta} r
    -\big[D^{\beta}(\sigma^{i}u_{x^i})
    -\sigma^{i}D^{\beta}u_{x^i}\big]\Big\}
    - \sigma^{i}D^{\beta}u_{x^i}\bigg)
    \bigg\|^{2}\,dx\\
    & \leq -2
    \int_{\bR^d} G''
    \bigg\| \sum_{|\beta| \leq m}
    \sigma^{i}D^{\beta}u D^{\beta} u_{x^i} \bigg\|^{2}\,dx\\
    & ~~~~ + 4\int_{\bR^d} G''
    \bigg[\sum_{|\beta| \leq m} (D^{\beta} u)\sigma^{ik} D^{\beta}u_{x^i}
    \bigg]
    \bigg(\sum_{|\gamma| \leq m} D^{\gamma}u D^{\gamma} r^{k}
    \bigg) \,dx\\
    & ~~~~ - 4 \int_{\bR^d} G''
    \bigg\{ \sum_{|\beta| \leq m}D^{\beta}u
    \big[D^{\beta}(\sigma^{ik}u_{x^i})-\sigma^{ik}D^{\beta}u_{x^i}\big]\bigg\}
    \bigg( \sum_{|\gamma| \leq m}\sigma^{jk}D^{\gamma}u D^{\gamma}u_{x^j}
    \bigg)\,dx.
  \end{split}
\end{eqnarray}
Then taking \eqref{ieq:c5}--\eqref{ieq:c8} into account and keeping
$\alpha^{ij} = \frac{1}{2}\sigma^{ik}\sigma^{jk}$ in mind, we obtain
\begin{eqnarray}\label{ieq:c9}
  \begin{split}
    R_1 + & R_2 + R_3 + R_4 \\
    \leq & - (1-\varepsilon)I + {C\over\varepsilon} J
    - \sum_{|\beta| \leq m} \int_{\bR^d} G'
    \|\sigma^{i}D^{\beta}u_{x^i}\|^{2}\,dx
    -2 \int_{\bR^d} G''
    \bigg\| \sum_{|\beta| \leq m}
    \sigma^{i}D^{\beta}u D^{\beta} u_{x^i} \bigg\|^{2}\,dx\\
    & + 2 \sum_{|\beta| \leq m} \sum_{\beta} \int_{\bR^d} G'
    \sigma^{jk}D^{\beta}u D^{\beta_1}
    \sigma^{ik} D^{\beta_2}u_{x^i x^j}\,dx\\
    = &  - (1-\varepsilon)I + {C\over\varepsilon} J
    - 2 \sum_{|\beta| \leq m} \int_{\bR^d} G'
    \alpha^{ij}D^{\beta}u_{x^i} D^{\beta} u_{x^j} \,dx\\
    & -4\int_{\bR^d} G'' \alpha^{ij}
    \bigg( \sum_{|\beta| \leq m} D^{\beta}u D^{\beta}u_{x^i} \bigg)
    \bigg( \sum_{|\gamma| \leq m} D^{\gamma}u D^{\gamma}u_{x^j} \bigg)\,dx \\
    &+ 2\sum_{|\beta| \leq m} \sum_{\beta}
    \int_{\bR^d} G' D^{\beta}u D^{\beta_1} \alpha^{ij}
    D^{\beta_2} u_{x^i x^j} \,dx.
  \end{split}
\end{eqnarray}
By \eqref{ieq:c1}--\eqref{ieq:c4}, \eqref{ieq:c9} and the parabolicity
condition ($\DP$) (i.e. the matrix $(a^{ij}-\alpha^{ij})_{d \times d} \geq
0$), we have
\begin{eqnarray}\label{ieq:c10}
  \begin{split}
    &\int_{\bR^d} \Theta(u,r,f,x,t,\omega) \,dx \\
    &~~~~ \leq  - (1-\varepsilon)I + (C+{C\over\varepsilon}) J
    - 2 \sum_{|\beta| \leq m} \int_{\bR^d} G'
    (a^{ij}-\alpha^{ij}) D^{\beta}u_{x^i} D^{\beta}
    u_{x^j} \,dx\\
    &~~~~~~~~ -4\int_{\bR^d} G'' (a^{ij}-\alpha^{ij})
    \bigg( \sum_{|\beta| \leq m} D^{\beta}u D^{\beta}u_{x^i} \bigg)
    \bigg( \sum_{|\gamma| \leq m} D^{\gamma}u D^{\gamma}u_{x^j} \bigg)\,dx \\
    &~~~~~~~~ + 2 \sum_{|\beta| \leq m} \sum_{\beta}
    \int_{\bR^d} G' D^{\beta}u D^{\beta_1}(a^{ij}-\alpha^{ij})
    D^{\beta_2} u_{x^i x^j} \,dx+\sum_{|\beta| \leq m}
    \int_{\bR^d} G'|D^{\beta}f|^{2} \,dx\\
    &~~~~ \leq  - (1-\varepsilon)I + (C+{C\over\varepsilon}) J
    - 2 \sum_{|\beta| \leq m}  \int_{\bR^d} G'
    A^{ij} D^{\beta}u_{x^i} D^{\beta}
    u_{x^j} \,dx\\
    &~~~~~~~~ + 2 \sum_{|\beta| \leq m} \sum_{\beta}
    \int_{\bR^d} G' D^{\beta}u D^{\beta_1}A^{ij}
    D^{\beta_2} u_{x^i x^j} \,dx+\sum_{|\beta| \leq m}
    \int_{\bR^d} G'|D^{\beta}f|^{2} \,dx,
  \end{split}
\end{eqnarray}
with $A^{ij} := a^{ij}-\alpha^{ij}$. 

To
proceed a further estimate to (\ref{ieq:c10}), we need the following well-known lemma.
\begin{lem}\label{lem:olenik} (Oleinik \cite{Olei67})
  Assume that $b^{ij}(x)\xi^{i}\xi^{j} \geq 0$ for all $x\in \bR^d$,
  $\xi = (\xi^1,\dots,\xi^d)\in \bR^d$ and $b^{ij}\in C^{2}(\bR^d)$. Then for any
  function $v\in C^{2}(\bR^d)$,
  \[
  \left(b^{ij}_{x^{\rho}}v_{x^i x^j}\right)^{2} \leq
  C'b^{ij} v_{x^i x^k} v_{x^j x^k},
  ~~~~\rho = 1,\dots,d,
  \]
  where $C'$ depends only on the second order derivatives of $b^{ij}$.
\end{lem}

By Lemma \ref{lem:olenik}, in view of the twice differentiability of $A^{ij}$ and the condition $(A^{ij})_{d \times d}\geq 0$, we have
\[
\begin{split}
  \left(D^{\beta_1}A^{ij} D^{\beta_2} u_{x^i x^j}\right)^{2}
  \leq C \sum_{|\gamma|=1} A^{ij} D^{\gamma + \beta_2}u_{x^i}
  D^{\gamma + \beta_2}u_{x^j}
  \leq C \sum_{|\beta| \leq m}A^{ij} D^{\beta} u_{x^i} D^{\beta}
  u_{x^j}.
\end{split}
\]
Thus,
\[
    2 \int_{\bR^d} G' D^{\beta}u D^{\beta_1}A^{ij}
    D^{\beta_2} u_{x^i x^j} \,dx \leq \eps
    \int_{\bR^d} G' \left(D^{\beta_1}A^{ij}
    D^{\beta_2} u_{x^i x^j}\right)^{2} \,dx
    + {C\over\varepsilon} J .
\]
Therefore, (\ref{ieq:c10})
can be simplified as below:
\[
  \begin{split}
    &\int_{\bR^d} \Theta(u,r,f; x,t,\omega) \,dx\\
    &~~~~ \leq  - (1-\varepsilon) I
    - (2-C\eps) \sum_{|\beta| \leq m} \int_{\bR^d} G'
    A^{ij} D^{\beta}u_{x^i} D^{\beta}
    u_{x^j} \,dx\\
    &~~~~~~~~ +  (C+{C\over\varepsilon}) J
    +\sum_{|\beta| \leq m} \int_{\bR^d} G'|D^{\beta}f|^{2} \,dx.
  \end{split}
\]
Noting the fact that $(A^{ij})_{d \times d}\geq 0$ and $\varepsilon$ can be chosen sufficiently small, we obtain the estimate
\eqref{est:basic} and the proof of Lemma \ref{lem:basic} is complete.

\section{Proof of Lemma \ref{lem:sm.eq.est}}

We need do some preparations before proving Lemma \ref{lem:sm.eq.est}. By condition (\B) and
Sobolev's imbedding theorem, there exists modifications of $f$ and $\vf$,
still denoted by $f$ and $\vf$, which are infinitely differentiable w.r.t.
$x$ for all $(\omega,t)$. Since the imbedding of $W^{m,2}$ in $C^n$ is
continuous when $2(m-n)>d$, the measurability is preserved. Therefore, $f$ are
$\sF_{t}$-measurable for each $(t,x)$ and predictable for each $x$.

In view of Theorem 2.3 in Du and Meng~\cite{DuMe10}, we know that under conditions ($\SP$) and (\B),
BSPDE~\eqref{eq:main}-\eqref{con:terminal} has a unique generalized solution $(u,q)$ satisfying
\begin{equation}\label{prp:sm.eq.sol}
u \in \bigcap_{m\in \mathbb{Z}_{+}}L^{2}_{\sP}CW^{m,2},~~~~ q \in
\bigcap_{m\in \mathbb{Z}_{+}}L^{2}_{\sP}W^{m,2}.
\end{equation}
We always choose appropriate versions of $u$ and $q$ such that
\begin{enumerate}
\item[(i)] for each
$\omega$, $u(\omega) \in C([0,T];W^{m,2})$ for all $m \in \mathbb{Z}_+$;
\item[(ii)]
for all $t$, $\E \|q(t)\|_{m,2}^2 < \infty$ for all $m \in \mathbb{Z}_+$.
\end{enumerate}
By Sobolev's imbedding theorem again, we have the following properties to
$(u,q)$:
\begin{enumerate}
\item[(i)]
$u(t,x)$ and $q(t,x)$ are $\sF_{t}$-measurable for each $(t,x)$;
\item[(ii)]
$u(t,x)$ is jointly continuous in $(t,x)$ for each $\omega$;
\item[(iii)]
$u(t,x)$ and $q(t,x)$ are infinitely differentiable w.r.t. $x$ for each
$(\omega,t)$ and all the derivatives of $u(t,x)$ are continuous in $(t,x)$
for each $\omega$.
\end{enumerate}
Moreover, we first have the following
\begin{lem}\label{lem:sm.eq.prp}
  Under conditions \emph{($\SP$)} and \emph{($\B$)},
  BSPDE~\eqref{eq:main}-\eqref{con:terminal} has
  a unique generalized solution $(u,q)$ satisfying \eqref{prp:sm.eq.sol} and
  \begin{enumerate}
    \item[\emph{(a)}]
    for each $x$, the
    equation
    \begin{eqnarray}\label{eq:pointwiseform}
    \begin{split}
      D^{\alpha}u(t,x) = D^{\alpha}\vf(x) + \int_{t}^{T}
      D^{\alpha}\big(a^{ij}u_{x^i x^j}+b^{i}u_{x^i}+cu\\
      +\sigma^{ik}q^{k}_{x^i}+\nu^{k}q^{k}+f \big)(s,x)\,ds
      -\int_{t}^{T} D^{\alpha}q^{k}(s,x)\,dW_{s}^{k}
    \end{split}
    \end{eqnarray}
    holds for all $t \in [0,T]$ and all multi-index $\alpha$ on a full-measure set independent of $(t,x)$;
    \item[\emph{(b)}]
    for arbitrary $m\in \mathbb{Z}_+$ and $p\geq 2$,
    $u(\omega)\in C([0,T];W^{m,p})$ for each $\omega$.
  \end{enumerate}
\end{lem}

\begin{proof}
  Assertion (a) can be derived immediately by (\ref{eq:d1}) and the exchange of operators $S_{\eps}$ and $D^{\alpha}$.
%
  Assertion (b) can be concluded easily by
  Sobolev's imbedding theorem. For this, recalling that
  $u(\omega) \in \bigcap_{n\geq 0}C([0,T];W^{n,2})$ for each
  $\omega$, by Sobolev's theorem we can find an $n$ such that $W^{n,2}$ is
  continuously imbedded into $W^{m,p}$. Therefore
  $u(\omega) \in C([0,T];W^{m,p})$ for each $\omega$ and
  the lemma is proved.
\end{proof}

\begin{proof}[Proof of Lemma \ref{lem:sm.eq.est}]
\emph{Step 1.} For positive integers $M,N$, we define
$$H_M(x)=x^2I_{\{-M\leq
x<M\}}+M(2x-M)I_{\{x\geq M\}}-M(2x+M)I_{\{x<-M\}}$$ and
$$G_{N,p}(x)=x^{p}I_{\{0\leq
x<N\}}+N^{p-1}[px-(p-1)N]I_{\{x\geq N\}}.$$ It is clear that $H_M \in
C^{1,1}(\R)$, $G_{N,p}\in C^{1,1}[0,\infty)$ satisfying $G(s),G'(s)>0$ and
$G''(s) \geq 0$ for all $s\in (0,\infty)$, and
\[
  H_M(x) \uparrow x^2,~~H_M'(x)\rrow 2x~; ~~~G_{N,p}(x)\uparrow x^p,~~
  G_{N,p}'(x)\uparrow p x^{p-1},
\]
as $M,N\rrow \infty$, respectively.

Moreover, $G_{N,p}$ satisfies that for all
$x \geq 0$,
\begin{equation}\label{001}
  xG_{N,p}'(x) \leq pG_{N,p}(x),~~~~
  |G_{N,p}'(x)|^{\frac{p}{p-1}} \leq p^{\frac{p}{p-1}}G_{N,p}(x).
\end{equation}
\smallskip

For simplicity, we denote
$$
  \Psi_{M} :=  \sum_{|\alpha|\leq
  m}H_M\left(D^\alpha u\right).
$$
Doing the generalized It$\hat {\rm o}$'s formula (c.f. \cite{ETZ07}) to
$${\rm e}^{Ks}G_{N,p}\big(\Psi_{M}\big),$$ for each $x$, by
\eqref{eq:trans} and \eqref{eq:pointwiseform} we have
\begin{eqnarray}\label{000}
&&{\rm e}^{Kt}G_{N,p}\big(\Psi_M\big)+K\int_t^T{\rm
e}^{Ks}G_{N,p}\big(\Psi_{M}\big)\,ds\nonumber\\
&=&{\rm e}^{KT}G_{N,p}\bigg(\sum_{|\alpha|\leq m}H_M(D^\alpha
\varphi)\bigg)+\int_t^T{\rm
e}^{Ks}G^{'}_{N,p}\big(\Psi_M\big)\sum_{|\beta|\leq m}H^{'}_M(D^\beta
u)D^{\beta}\big(\mathbb{D}u\big)\,ds\nonumber\\
&&-\int_t^T{\rm e}^{Ks}G^{'}_{N,p}\big(\Psi_M\big)\sum_{|\beta|\leq
m}I_{\{-M\leq D^\beta u<M\}}\|D^{\beta}(r-\sigma^{i}u_{x^i})\|^{2}\,ds\nonumber\\
&&-{1\over2}\int_t^T{\rm
e}^{Ks}G^{''}_{N,p}\big(\Psi_M\big)\bigg\|\sum_{|\beta|\leq m}H^{'}_M(D^\beta
u)D^{\beta}(r^{k}-\sigma^{ik}u_{x^i})\bigg\|^2\,ds\nonumber\\
&&-\int_t^T{\rm e}^{Ks}G^{'}_{N,p}\big(\Psi_M)\sum_{|\beta|\leq
m}H^{'}_M(D^\beta u)D^{\beta}(r^{k}-\sigma^{ik}u_{x^i})dW_s^k,
\end{eqnarray}
where we denote
$$\mathbb{D}u:=(a^{ij}-2\alpha^{ij})u_{x^ix^j}+\widetilde{b}^{i}u_{x^i}+cu
+\sigma^{ik}r^{k}_{x^i}+\nu^{k}r^{k}+f.$$ Integrating both sides of
\eqref{000} w.r.t. $x$ in $\R^d$ and noting that
$$G^{'}_{N,p}\big(\Psi_{M}\big)\sum_{|\beta|\leq m}H^{'}_M(D^\beta u)$$ are
bounded, by the stochastic Fubini theorem we have
\begin{eqnarray}\label{002}
&&\int_{\R^d}{\rm
e}^{Kt}G_{N,p}\big(\Psi_{M}\big)\,dx+K\int_t^T\!\!\!\int_{\R^d}{\rm
e}^{Ks}G_{N,p}\big(\Psi_{M}\big)\,ds\nonumber\\
&=&\int_{\R^d}{\rm e}^{KT}G_{N,p}\bigg(\sum_{|\alpha|\leq m}H_M(D^\alpha
\varphi)\bigg)\,dx\nonumber\\
&&+\int_t^T\!\!\!\int_{\R^d}{\rm
e}^{Ks}G^{'}_{N,p}\big(\Psi_{M}\big)\sum_{|\beta|\leq m}H^{'}_M(D^\beta
u)D^{\beta}\big(\mathbb{D}u\big)\,dxds\nonumber\\
&&-\int_t^T\!\!\!\int_{\R^d}{\rm
e}^{Ks}G^{'}_{N,p}\big(\Psi_{M}\big)\sum_{|\beta|\leq
m}I_{\{-M\leq D^\beta u<M\}}\|D^{\beta}(r-\sigma^{i}u_{x^i})\|^{2}\,dxds\nonumber\\
&&-{1\over2}\int_t^T\!\!\!\int_{\R^d}{\rm
e}^{Ks}G^{''}_{N,p}\big(\Psi_{M}\big)\bigg\|\sum_{|\beta|\leq
m}H^{'}_M(D^\beta
u)D^{\beta}(r^{k}-\sigma^{ik}u_{x^i})\bigg\|^2\,dxds\nonumber\\
&&-\int_t^T\!\!\!\int_{\R^d}{\rm
e}^{Ks}G^{'}_{N,p}\big(\Psi_{M}\big)\sum_{|\beta|\leq m}H^{'}_M(D^\beta
u)D^{\beta}(r^{k}-\sigma^{ik}u_{x^i})\,dxdW_s^k.
\end{eqnarray}
First note that $G_{N,p}(\Psi_{M})\leq N^{p-1}p\Psi$, hence
$|\int_{\R^d}G_{N,p}(\Psi_{M})\,dx|$ is bounded. Similarly,
$\int_{\R^d}G_{N,p}(\sum_{|\alpha|\leq m}H_M(D^\alpha \varphi))\,dx$ is
bounded. On the other hand, both $G^{'}_{N,p}(\Psi_{M})$ and $H^{'}_M(D^\beta
u)$ are bounded, thus the above stochastic integral is a martingale. Taking
first the expectation and then the limit of $M$ for all terms in (\ref{002}),
we have
\begin{eqnarray*}
&&\E\int_{\R^d}{\rm
e}^{Kt}G_{N,p}\big(\Psi\big)\,dx+K\E\int_t^T\!\!\!\int_{\R^d}{\rm
e}^{Ks}G_{N,p}\big(\Psi\big)\,dxds\nonumber\\
&=&\E\int_{\R^d}{\rm e}^{KT}G_{N,p}\bigg(\sum_{|\alpha|\leq
m}|D^\alpha\varphi|^2\bigg)\,dx\nonumber\\
&&+2\E\int_t^T\!\!\!\int_{\R^d}{\rm
e}^{Ks}G^{'}_{N,p}\big(\Psi\big)\sum_{|\beta|\leq m}D^\beta u
D^{\beta}\big(\mathbb{D}u\big)\,dxds\nonumber\\
&&-\E\int_t^T\!\!\!\int_{\R^d}{\rm
e}^{Ks}G^{'}_{N,p}\big(\Psi\big)\sum_{|\beta|\leq
m}\|D^{\beta}(r-\sigma^{i}u_{x^i})\|^{2}\,dxds\nonumber\\
&&-\E\int_t^T\!\!\!\int_{\R^d}{\rm
e}^{Ks}G^{''}_{N,p}\big(\Psi\big)\bigg\|\sum_{|\beta|\leq m}D^\beta
uD^{\beta}(r^{k}-\sigma^{ik}u_{x^i})\bigg\|^2\,dxds.
\end{eqnarray*}
Thus it follows from Lemma \ref{lem:basic} that
\begin{eqnarray*}
&&\E\int_{\R^d}{\rm
e}^{Kt}G_{N,p}\big(\Psi\big)\,dx+K\E\int_t^T\!\!\!\int_{\R^d}{\rm
e}^{Ks}G_{N,p}\big(\Psi\big)\,dxds\nonumber\\
&=&\E\int_{\R^d}{\rm e}^{KT}G_{N,p}\bigg(\sum_{|\alpha|\leq
m}|D^\alpha\varphi|^2\bigg)\,dx+\E\int_t^T\!\!\!\int_{\R^d}{\rm
e}^{Ks}\Theta(u,r,f,x,s,\omega)\,dxds\nonumber\\
&\leq&\E\int_{\R^d}{\rm e}^{KT}G_{N,p}\bigg(\sum_{|\alpha|\leq
m}|D^\alpha\varphi|^2\bigg)\,dx -(1-\eps)\E\int_t^T\!\!\!\int_{\R^d}{\rm
e}^{Ks}G^{'}_{N,p}(\Psi)\Upsilon \,dxds\nonumber\\
&&+{C\over\eps}\E\int_t^T\!\!\!\int_{\R^d}{\rm
e}^{Ks}G_{N,p}(\Psi)\,dxds+{C\over\eps}\E\int_t^T\!\!\!\int_{\R^d}{\rm
e}^{Ks}G^{'}_{N,p}(\Psi)\Psi
\,dxds\nonumber\\
&&+C\E\int_t^T\!\!\!\int_{\R^d}{\rm
e}^{Ks}G^{'}_{N,p}(\Psi)\sum_{|\beta| \leq m}|D^{\beta}f|^{2}\,dxds,\nonumber
\end{eqnarray*}
where the function $\Theta$ is defined by \eqref{formula:Theta} w.r.t. $G_{N,p}$ instead of $G$. Then, in view of \eqref{001} and by Young's
inequality, we have
\begin{eqnarray}\label{008}
&&\E\int_{\R^d}{\rm e}^{Kt}G_{N,p}(\Psi)\,dx
+(K-{C\over\eps}-{{Cp}\over\eps}-C)\E\int_t^T\!\!\!\int_{\R^d}{\rm
e}^{Ks}G_{N,p}(\Psi)\,dxds\nonumber\\
&&+(1-\eps)\E\int_t^T\!\!\!\int_{\R^d}{\rm e}^{Ks}G^{'}_{N,p}(\Psi)\Upsilon
\,dxds\nonumber\\
&\leq&\E\int_{\R^d}{\rm e}^{KT}G_{N,p}\big(\sum_{|\alpha|\leq m}|D^\alpha
\varphi|^2\big)\,dx+C{\rm
e}^{KT}({{p-1}\over p})^{p-1}\E\int_t^T\|f\|_{m,2p}^{2p}\,ds.
\end{eqnarray}
Choosing
$$\eps={1\over2}\quad \hbox{ \rm and }\quad K = 3C+2Cp+1,$$
 letting $N\to \infty$ in
(\ref{008}) and using the monotone convergence, we have
\begin{eqnarray}\label{010}
&&\E\int_0^T\!\!\!\int_{\R^d}\Psi^p\,dxds
+\E\int_0^T\!\!\!\int_{\R^d}\Psi^{p-1}\Upsilon
\,dxds\nonumber\\
&\leq&C{\rm e}^{Cp}\E\|\varphi\|_{m,2p}^{2p}+C{\rm
e}^{Cp}\E\int_0^T\|f\|_{m,2p}^{2p}\,ds<\infty.
\end{eqnarray}
\medskip

\emph{Step 2.} We estimate the term $\E\sup_{t\leq
T}\int_{\R^d}\Psi^p\,dx$. By the BDG inequality, it turns out that
\begin{eqnarray*}
&&\E\sup_{t\leq T}\bigg|\int_0^T\!\!\!\int_{\R^d}
G^{'}_{N,p}\big(\Psi_{M}\big)\sum_{|\beta|\leq m}H^{'}_M(D^\beta
u)D^{\beta}(r^{k}-\sigma^{ik}u_{x^i})\,dxdW_s^k\bigg|\nonumber\\
&\leq & C\E\bigg[\bigg\{\int_0^T\bigg\|\int_{\R^d}
G^{'}_{N,p}\big(\Psi_{M}\big)\sum_{|\beta|\leq m}H^{'}_M(D^\beta
u)D^{\beta}(r-\sigma^{i}u_{x^i})\,dx\bigg\|^2\,ds\bigg\}^{1\over2}\bigg]
\nonumber\\
&\leq&C\E\bigg[\bigg\{\int_0^T\bigg(\int_{\R^d}
G^{'}_{N,p}\big(\Psi_{M}\big)\sum_{|\beta|\leq m}|H^{'}_M(D^\beta
u)|\sum_{|\gamma|\leq m}\|D^{\gamma}r\|\,dx\bigg)^2\nonumber\\
&&\ \ \ \ \ \ \ \ \ \ \ \ \ \ +\bigg(\int_{\R^d}
G^{'}_{N,p}\big(\Psi_{M}\big)\sum_{|\beta|\leq m}|H^{'}_M(D^\beta
u)|\sum_{|\gamma|\leq
m}|D^{\gamma}u|\,dx\bigg)^2\nonumber\\
&&\ \ \ \ \ \ \ \ \ \ \ \ \ \ \ \ \ \ +\bigg(\int_{\R^d}
G^{'}_{N,p}\big(\Psi_{M}\big)\sum_{|\beta|\leq m}H^{'}_M(D^\beta
u)\sigma^{ik}D^{\beta}u_{x^i}\,dx\bigg)^2\,ds\bigg\}^{1\over2}\bigg]\nonumber\\
&\leq&C\E\bigg[\bigg\{\int_0^T\bigg(\int_{\R^d}
G^{'}_{N,p}\big(\Psi_{M}\big)\Psi_M \,dx \int_{\R^d} G^{'}_{N,p}
\big(\Psi_{M}\big)\Upsilon \,dx\nonumber\\
&&\ \ \ \ \ \ \ \ \ \ \ \ \ \ \ \ +\int_{\R^d} G^{'}_{N,p}\big(\Psi_{M}\big)\Psi_M
\,dx \int_{\R^d} G^{'}_{N,p}\big(\Psi_{M}\big)\Psi \,dx\nonumber\\
&&\ \ \ \ \ \ \ \ \ \ \ \ \ \ \ \ +\bigg|\int_{\R^d}
\{G_{N,p}\big(\Psi_{M}\big)\}_{x^i}\sigma^{ik}\,dx\bigg|^2\bigg)\,ds
\bigg\}^{1\over2}\bigg]\nonumber\\
&\leq&C\E\sqrt{\bigg(\sup_{t\leq T}\int_{\R^d}
G^{'}_{N,p}\big(\Psi_{M}\big)\Psi_{M} \,dx\bigg) \int_0^T\!\!\!\int_{\R^d}
G^{'}_{N,p}\big(\Psi_{M}\big)\Upsilon \,dxds}
\nonumber\\
&&+C\E\sqrt{\bigg(\sup_{t\leq T}\int_{\R^d}
G^{'}_{N,p}\big(\Psi_{M}\big)\Psi_{M} \,dx\bigg) \int_0^T\!\!\!\int_{\R^d}
G^{'}_{N,p}\big(\Psi_{M}\big)\Psi \,dxds}
\nonumber\\
&&+C\E\sqrt{\bigg(\sup_{t\leq T}\int_{\R^d}
G_{N,p}\big(\Psi_{M}\big)\,dx\bigg)\int_0^T\!\!\!\int_{\R^d} G_{N,p}
\big(\Psi_{M}\big)\,dxds}\nonumber\\
&\leq&{\eps\over p} \E\sup_{t\leq
T}\int_{\R^d}G^{'}_{N,p}\big(\Psi_{M}\big)\Psi_{M}
\,dx+Cp\E\int_0^T\!\!\!\int_{\R^d} G^{'}_{N,p}
\big(\Psi_{M}\big)\Upsilon \,dxds\nonumber\\
&&+{\eps\over p} \E\sup_{t\leq
T}\int_{\R^d}G^{'}_{N,p}\big(\Psi_{M}\big)\Psi_{M}
\,dx+Cp\E\int_0^T\!\!\!\int_{\R^d} G^{'}_{N,p}
\big(\Psi_{M}\big)\Psi \,dxds\nonumber\\
&&+\eps \E\sup_{t\leq
T}\int_{\R^d}G_{N,p}\big(\Psi_{M}\big)\,dx+C\E\int_0^T\!\!\!\int_{\R^d}
G_{N,p}\big(\Psi_{M}\big)\,dxds.
\end{eqnarray*}
Hence, Considering (\ref{002}) with $K=0$, by BDG inequality
and \eqref{001} we have
\begin{eqnarray*}
&&(1-3\eps)\E\sup_{t\leq T}\int_{\R^d}G_{N,p}\big(\Psi_{M}\big)\,dx\nonumber\\
&\leq&\E\int_{\R^d}G_{N,p}\bigg(\sum_{|\alpha|\leq m}H_M(D^\alpha
\varphi)\bigg)\,dx\nonumber\\
&&+\E\int_0^T\bigg|\int_{\R^d} G^{'}_{N,p}\big(\Psi_{M}\big)\sum_{|\beta|\leq
m}H^{'}_M(D^\beta
u)D^{\beta}\big(\mathbb{D}u\big)\,dx\nonumber\\
&&\ \ \ \ \ \ \ \ \ \ \ \ \ \ -\int_{\R^d}
G^{'}_{N,p}\big(\Psi_{M}\big)\sum_{|\beta|\leq
m}I_{\{-M\leq D^\beta u<M\}}\|D^{\beta}(r-\sigma^{i}u_{x^i})\|^{2}\,dx\nonumber\\
&&\ \ \ \ \ \ \ \ \ \ \ \ \ \ -{1\over2}\int_{\R^d}
G^{''}_{N,p}\big(\Psi_{M}\big)\Big\|\sum_{|\beta|\leq m}H^{'}_M(D^\beta
u)D^{\beta}(r^{k}-\sigma^{ik}u_{x^i})\Big\|^2\,dx\bigg|\,ds\nonumber\\
&&+Cp\E\int_0^T\!\!\!\int_{\R^d} G^{'}_{N,p}\big(\Psi_{M}\big)\Upsilon \,dxds
+Cp\E\int_0^T\!\!\!\int_{\R^d} G^{'}_{N,p}\big(\Psi_{M}\big)\Psi
\,dxds\\
&&+C\E\int_0^T\!\!\!\int_{\R^d} G_{N,p}\big(\Psi_{M}\big)\,dxds.\nonumber
\end{eqnarray*}
Taking the limit of $M$, by Lemma \ref{lem:basic} and \eqref{001} again we have
\begin{eqnarray}\label{012}
&&(1-3\eps)\E\sup_{t\leq T}\int_{\R^d}G_{N,p}\big(\Psi\big)\,dx\nonumber\\
&\leq&\E\int_{\R^d}G_{N,p}\bigg(\sum_{|\alpha|\leq m}|D^\alpha
\varphi|^2\bigg)\,dx+\E\int_0^T\bigg|\int_{\R^d}
\Theta(u,r,f,x,s,\omega)\,dx\bigg|\,ds\nonumber\\
&&+Cp\E\int_0^T\!\!\!\int_{\R^d}G^{'}_{N,p}\big(\Psi\big)\Upsilon
\,dxds+Cp\E\int_0^T\!\!\!\int_{\R^d}G^{'}_{N,p}\big(\Psi\big)\Psi
\,dxds\nonumber\\
&&+C\E\int_0^T\!\!\!\int_{\R^d}G_{N,p}\big(\Psi\big)\,dxds\nonumber\\
&\leq&\E\int_{\R^d}G_{N,p}\bigg(\sum_{|\alpha|\leq m}|D^\alpha
\varphi|^2\bigg)\,dx+\E\int_0^T\!\!\!\int_{\R^d}\sum_{|\beta| \leq
m}G'(\Psi)|D^{\beta}f|^{2}\,dxds\nonumber\\
&&+(Cp+1)\E\int_0^T\!\!\!\int_{\R^d}G^{'}_{N,p}\big(\Psi\big)\Upsilon
\,dxds+({C\over
\eps}+Cp)\E\int_0^T\!\!\!\int_{\R^d}G^{'}_{N,p}\big(\Psi\big)\Psi
\,dxds\nonumber\\
&&+({C\over\eps}+C)\E\int_0^T\!\!\!\int_{\R^d}G_{N,p}\big(\Psi\big)\,dxds\nonumber\\
&\leq&C{\rm e}^{Cp}\E\|\varphi\|_{m,2p}^{2p}+C{\rm e}^{Cp}
\E\int_t^T\|f\|_{m,2p}^{2p}\,ds
+(Cp+1)\E\int_0^T\!\!\!\int_{\R^d}G^{'}_{N,p}\big(\Psi\big)\Upsilon
\,dxds\nonumber\\
&&+({Cp\over
\eps}+Cp^{2}+{C\over\eps}+C)\E\int_0^T\!\!\!\int_{\R^d}G_{N,p}\big(\Psi\big)\,dxds.
\end{eqnarray}
Then choosing $\eps = \frac{1}{6}$ and letting $N\to \infty$ in (\ref{012}), in view of (\ref{010}), we have
\begin{eqnarray*}\label{015}
\E\sup_{t\leq T}\int_{\R^d}\Psi^p\,dx &\leq&C {\rm
e}^{Cp}\E\|\varphi\|_{m,2p}^{2p} +C {\rm e}^{Cp}
\E\int_0^T\|f\|_{m,2p}^{2p}\,ds\nonumber\\
&&+Cp^{2}\E\int_0^T\!\!\!\int_{\R^d}\Psi^p\,dxds
+Cp\E\int_0^T\!\!\!\int_{\R^d}\Psi^{p-1}\Upsilon \,dxds\nonumber\\
&\leq&C{\rm e}^{Cp}\E\|\varphi\|_{m,2p}^{2p}+C{\rm
e}^{Cp}\E\int_0^T\|f\|_{m,2p}^{2p}\,ds<\infty.
\end{eqnarray*}
The proof of Lemma \ref{lem:sm.eq.est} is complete.
\end{proof}

\bibliographystyle{model1b-num-names}

\end{document}